\documentclass[11pt]{amsart}
\usepackage[margin=1.12in]{geometry}
\usepackage{amsmath,amssymb,amsthm,mathtools}
\usepackage{mathrsfs}
\usepackage{enumitem}
\usepackage{tikz-cd}
\usepackage{csquotes}
\usepackage[backend=biber,style=alphabetic,maxbibnames=99,maxcitenames=5,giveninits=true]{biblatex}
\usepackage[colorlinks=true,citecolor=blue,linkcolor=red,urlcolor=blue]{hyperref}

\numberwithin{equation}{section}

\newtheorem{theorem}{Theorem}[section]
\newtheorem{proposition}[theorem]{Proposition}
\newtheorem{lemma}[theorem]{Lemma}
\newtheorem{corollary}[theorem]{Corollary}

\newtheorem{mainthm}{Theorem}
\newtheorem{maincor}[mainthm]{Corollary}

\theoremstyle{definition}
\newtheorem{definition}[theorem]{Definition}

\newtheorem{example}[theorem]{Example}
\newtheorem{remark}[theorem]{Remark}

\theoremstyle{remark}

\newcommand{\A}{\mathsf A}
\newcommand{\K}{\mathsf K}
\newcommand{\G}{\mathcal G}
\newcommand{\Hh}{\mathcal H}
\newcommand{\B}{\mathcal B}
\newcommand{\F}{\mathcal F}
\newcommand{\T}{\mathcal T}
\newcommand{\Q}{\mathbb Q}
\newcommand{\Z}{\mathbb Z}
\newcommand{\PP}{\mathbb P}
\newcommand{\LL}{\mathcal L}
\newcommand{\OO}{\mathcal O}
\newcommand{\rk}{\operatorname{rk}}

\newcommand{\cone}{\operatorname{cone}}
\newcommand{\Star}{\operatorname{Star}}
\newcommand{\Fac}{\operatorname{Fac}}
\newcommand{\ch}{\operatorname{ch}}
\newcommand{\td}{\operatorname{td}}
\newcommand{\Td}{\operatorname{Td}}
\newcommand{\Bl}{\operatorname{Bl}}

\newcommand{\Gmax}{\mathcal G_{\max}}

\newcommand{\Hy}{\mathfrak H_y}

\title[Tangent classes for matroid building sets]{Tangent classes for matroid building sets}
\author{Ronnie Cheng}
\address{Department of Mathematics, Stanford University, USA}
\email{rtcheng@stanford.edu}
\date{\today}

\begin{document}

\begin{abstract}
Significant combinatorial constraints and structure on matroids can come from notions in algebraic geometry, even without the matroids themselves being representable.

Let \(M\) be a loopless matroid on a finite ground set \(E\), and let \(\G\) be a building set containing the top flat \(E\). We define a tangent class \(T_{M,\G}\) in the \(K\)-ring \(K(M,\G)\), which extends the tangent bundle class of the de Concini--Procesi wonderful model from realizable matroids to arbitrary matroids with building sets.

The class \(T_{M,\G}\) satisfies a matroidal Hirzebruch--Riemann--Roch package. More precisely, its Hirzebruch class
\[
\operatorname{ch}(\lambda_y T_{M,\G}^{\vee})\operatorname{td}(T_{M,\G})
\]
specializes to the Todd class and computes the Chow polynomial of \((M,\G)\). In the realizable case, these identities agree with the usual tangent-bundle computations on the corresponding wonderful model.

As an application, we prove Chern-number inequalities for \(T_{M,\G}\), including a Miyaoka--Yau type inequality with respect to the hyperplane class.
\end{abstract}

\maketitle
\tableofcontents
\section{Introduction}

Building sets are combinatorial data that govern wonderful compactifications.
In the de Concini--Procesi construction \cite{DCP}, the choice of a building
set specifies which strata of an arrangement are successively resolved;
changing this choice gives different smooth compactifications of the same
arrangement complement.  This point of view encompasses several familiar
compactifications.  For example, the Fulton--MacPherson compactification of
configuration spaces can be viewed, in the language of wonderful
compactifications, as the model associated to a natural building set of
diagonals \cite{FultonMacPherson}. 

A key feature of these models is that much of their intersection theory is combinatorial.  If a matroid \(M\) is realized by a linear subspace
\(L\subseteq \Bbbk^E\), then the Chow ring, and likewise the \(K\)-ring, of
the corresponding wonderful model admit presentations depending only on the
linear-dependence data encoded by \(M\) and on the chosen building set, rather
than on the particular realization \(L\).  Consequently, the same
presentations make sense for arbitrary matroids, even when no realization or wonderful compactification exists.  This is the
guiding principle behind the matroidal Hodge-theoretic program initiated in
\cite{AHK}: one can extract algebro-geometric structures from realizable
arrangements, reformulate them purely in terms of matroids, and then ask
which geometric theorems continue to hold in this combinatorial setting.  The
purpose of the present paper is to pursue this philosophy for arbitrary
building sets and for the structure that, in the realizable case, comes from
the tangent bundle of a wonderful model.

Let \(M\) be a loopless matroid of rank \(r\) on a finite ground set \(E\),
and let \(\G\) be a building set that contains the top flat \(E\).  Associated
to \((M,\G)\) is the nested-set fan \(\Sigma_{M,\G}\) and the corresponding
smooth toric variety
\[
   X_{M,\G}:=X_{\Sigma_{M,\G}}.
\]
When \(M\) is realizable by a linear subspace \(L\subseteq \Bbbk^E\), the
building set \(\G\) also determines the de Concini--Procesi wonderful model
\(W_{L,\G}\) \cite{DCP}, which is a closed subvariety of \(X_{M,\G}\).
Moreover, the Feichtner--Yuzvinsky presentation identifies the Chow ring of
\(X_{M,\G}\) with the Chow ring of \(W_{L,\G}\) \cite{FY}, and the analogous
statement for \(K\)-rings is proved in \cite{LLPP}.  We write
\[
   \A(M,\G)
   \qquad\text{and}\qquad
   \K(M,\G)
\]
for these common Chow and \(K\)-rings.

For the maximal building set \(\Gmax\) (cf. Section~\ref{subsec:FY-rings}), a
tangent class
\[
   T_M\in \K(M,\Gmax)
\]
was constructed in \cite{ChengTangent}.  In the realizable case, this class
recovers the tangent bundle of the maximal de Concini--Procesi wonderful
model.  For arbitrary matroids, it realizes the Todd and Hirzebruch classes
appearing in the matroidal Hirzebruch--Riemann--Roch formula.  The purpose
of this paper is to extend this construction from \(\Gmax\) to every
building set \(\G\).

We first construct the Hirzebruch--Riemann--Roch package formally, without
assuming the existence of a tangent class.  Since the nested-set fan for
\(\Gmax\) refines the nested-set fan for \(\G\), there is a proper toric
morphism
\[
   \rho_{\Gmax,\G}:X_{M,\Gmax}\longrightarrow X_{M,\G}.
\]
One may therefore define the Euler characteristic on \((M,\G)\) by pulling
back to the maximal building set:
\[
   \chi_{M,\G}(\xi)
   :=
   \chi_{M,\Gmax}\bigl((\rho_{\Gmax,\G})^*\xi\bigr),
   \quad \text{for all }
   \xi\in\K(M,\G).
\]
By Poincar\'e duality for Chow rings of matroids with building sets
\cite{ADH23,PP}, there is a unique class
\[
   \Td_{M,\G}\in \A(M,\G)_\Q
\]
characterized by the Hirzebruch--Riemann--Roch formula
\[
   \chi_{M,\G}(\xi)
   =
   \deg_{M,\G}\bigl(\ch(\xi)\Td_{M,\G}\bigr),
   \qquad
   \xi\in\K(M,\G), 
\] where the Chern character map \(\ch:\K(M,\G)_\Q\xrightarrow{\sim}\A(M,\G)_\Q \) is an isomorphism since the variety \(X_{M, \G}\) is smooth.

The projection formula shows that the Todd class is given by the
pushforward
\[
   \Td_{M,\G}
   =
   (\rho_{\Gmax,\G})_*\td(T_{M,\Gmax}),
\]
where \(T_{M,\Gmax}=T_M\).

Similarly, by descending from
\(\Gmax\) to \(\G\) one flat at a time, one defines a Hirzebruch class 
\[
   \Hy(M,\G)\in \A(M,\G)_\Q[y].
\]
Its construction is given in Section~\ref{sec:formal-HRR-and-realization};
its specialization at \(y=0\) is the formal Todd class:
\[
   \Hy(M,\G)|_{y=0}=\Td_{M,\G}.
\]

This class satisfies the degree identity
\[
   \deg_{M,\G}\bigl(\Hy(M,\G)|_{y=-t}\bigr)
   =
   H_M^\G(t):=\sum_p\dim_\Q\A^p(M,\G)t^p.
\]
The substitution \(y=-t\) mirrors the classical fact that, for a smooth
projective variety whose cohomology is concentrated in Hodge bidegrees
\((p,p)\), the \(\chi_y\)-genus becomes the Chow polynomial after setting
\(y=-t\); this comparison is recalled in Section~\ref{sec:formal-HRR-and-realization}. 

This provokes the question: is this formal package realized by some meaningful tangent class?  We answer this affirmatively.  The construction is guided by the behavior of ordinary tangent bundles under blow-ups.  Namely, let
\[
   \rho:\widetilde X=\Bl_ZX\longrightarrow X
\]
be the blow-up of a smooth variety along a smooth complete intersection
\[
   Z=D_1\cap\cdots\cap D_\ell,
\]
with exceptional divisor \(E_{\mathrm{exc}}\) and strict transforms
\(\widetilde D_i\).  Aluffi's formula \cite{AluffiBlowups} gives the
following correction term in \(K\)-theory:
\[
   T_{\widetilde X}-\rho^*T_X
   =
   [\OO(E_{\mathrm{exc}})]-[\OO]
   +
   \sum_{i=1}^{\ell}
      \left(
         [\OO(\widetilde D_i)]
         -
         [\OO(\widetilde D_i+E_{\mathrm{exc}})]
      \right).
\]
When \(M\) is realizable by a linear subspace \(L\), this is exactly the type
of correction term that appears at each step of the de Concini--Procesi
iterated blow-up construction of \(W_{L,\G}\) starting from \(\PP(L)\). For an arbitrary matroid, there may be no variety \(W_{L,\G}\), but the same
one-step formula has a purely combinatorial interpretation in the \(K\)-rings.  We therefore use it to define, for every
building set \(\G\), a class
\[
   T_{M,\G}\in \K(M,\G)
\] even when \(M\) is not realizable by any field.

The theorem below says that this class realizes the formal Todd and
Hirzebruch classes in the same way that an ordinary tangent bundle does.

\begin{mainthm}[Theorem~\ref{thm:tangent-realization}]
For every loopless matroid \(M\) on ground set \(E\) and every building set
\(\G\) that contains the top flat \(E\),
\[
   \Hy(M,\G)
   =
   \ch(\lambda_yT_{M,\G}^{\vee})\td(T_{M,\G})
   \qquad\text{in }\A(M,\G)_\Q[y].
\]
In particular, specializing at \(y=0\) gives
\[
   \Td_{M,\G}=\td(T_{M,\G}),
\]
and applying \(\deg_{M,\G}\) after the substitution \(y=-t\) gives
\[
   \deg_{M,\G}
   \left(
      \ch(\lambda_yT_{M,\G}^\vee)\td(T_{M,\G})
   \right)\bigg|_{y=-t}
   =
   H_M^\G(t).
\]
Thus the Chow polynomial is recovered as the Hirzebruch genus of the tangent
class \(T_{M,\G}\).
\end{mainthm}

The proof has two main ingredients.  First, the intrinsic definition of
\(T_{M,\G}\) satisfies the relative Aluffi formula for every one-step
enlargement of building sets.  Second, the restriction of \(T_{M,\G}\) to
the one-step blow-up center satisfies the star-normal identity
\[
   \iota_F^*T_{M,\B}
   =
   T_{M|F,\B|F}\boxplus T_{M/F,\B/F}+N_F.
\]
Together these two identities allow the universal blow-up formula for
Hirzebruch classes to match exactly the recursive definition of
\(\Hy(M,\G)\).  The proof is entirely formal in the Chow and \(K\)-rings, and
therefore does not require a realization of \(M\).

The realization theorem also gives concrete numerical consequences.  In the
final subsection, we use the tangent class to identify the canonical class,
prove a formal Serre duality statement, and derive several Chern-number
identities.  Let \( \alpha\in \A^1(M,\G)\) denote the hyperplane class, and let \(d=r-1\).  The numbers
\[
   \int_{M,\G} c_k(T_{M,\G})\alpha^{d-k}
\]
are the matroidal analogues of polarized Chern numbers.  The main numerical
consequence is the following comparison with projective space.  Indeed, for
\(\PP^d\), \[
   \int_{\PP^d}c_k(T_{\PP^d})\alpha^{d-k}
   =
   \binom{d+1}{k}.
\]

\begin{maincor}[Corollary~\ref{cor:Chern-alpha-positive}]
Let \(d=r-1\).  For every \(0\le k\le d\),
\[
   \int_{M,\G}c_k(T_{M,\G})\alpha^{d-k}
   \ge
   \binom{d+1}{k}.
\]
\end{maincor}

Thus the matroidal tangent class has polarized Chern numbers at least as
large as those of projective space. 

Together with the Todd-class identities obtained from formal Hirzebruch--Riemann--Roch, the case \(k=2\) gives a Miyaoka--Yau type inequality with respect to the hyperplane class:
\[
   d
   \int_{M,\G}
      c_1(T_{M,\G})^2\alpha^{d-2}
   \le
   2(d+1)
   \int_{M,\G}
      c_2(T_{M,\G})\alpha^{d-2}.
\]
This has the same numerical form as the higher-dimensional Miyaoka--Yau
inequality with respect to a polarization; compare
\cite[Theorem~1.3]{GrebKebekusTaji}.

\subsection*{Organization of the paper}

Section~\ref{sec:prelim} recalls the necessary background on building sets,
nested-set fans, Feichtner--Yuzvinsky rings, de Concini--Procesi wonderful
models, and one-step star centers.  Section~\ref{sec:Aluffi-definition}
defines the tangent class \(T_{M,\G}\) and proves its relative Aluffi
formula.  Section~\ref{sec:formal-HRR-and-realization} constructs the formal
Hirzebruch--Riemann--Roch package and proves that it is realized by
\(T_{M,\G}\).  The final subsection records several corollaries, including the
canonical class formula, formal Serre duality, Chern-number inequalities,
and the Miyaoka--Yau type inequality with respect to \(\alpha\).

After the results of this paper had been obtained, but before the present
draft was posted, the problem was used as a case study in an experiment on
AI-assisted mathematical reasoning.  The resulting companion drafts present
the solution generated in that experiment and document the AI system used to
produce it \cite{DanusTangentClasses,Danus}.

\section*{Acknowledgements}

The author is grateful to Matt Larson for suggesting the strategy of pushing forward from the maximal building set.  The author
also thanks Shiyue Li for asking whether the tangent-class construction
extends from the maximal building set to arbitrary building sets.
\section{Preliminaries}\label{sec:prelim}

Throughout the paper, \(M\) is a loopless matroid on a finite ground set
\(E\), and \(r:=\rk(E)\).  We write \(\LL(M)\) for the lattice of flats of
\(M\), with \(\hat 0=\emptyset\) and \(\hat 1=E\). 

\subsection{Building sets and nested-set fans}\label{subsec:buildingset}

A subset
\[
   \G\subseteq \LL(M)\setminus\{\emptyset\}
\]
is a \emph{building set} if, for every nonempty flat \(F\), the maximal elements of
\(\G\) contained in \(F\), say
\[
   G_1,\ldots,G_s,
\]
induce the product decomposition
\[
   [\hat 0,F]
   \simeq
   [\hat 0,G_1]\times\cdots\times[\hat 0,G_s].
\]
These maximal elements are called the \(\G\)-factors of \(F\), and we denote
this set by
\[
   \Fac_{\G}(F)=\{G_1,\ldots,G_s\}.
\]
This is the standard building-set notion for lattices; see
\cite[Definition~1]{FY} and \cite[Definition~2.1]{FM}. In this paper, we additionally require the top flat \(\hat 1=E\) to be part of the building set, and we assume all building sets in the paper contain \(E\).

A subset \(\mathcal S\subseteq \G\setminus\{E\}\) is \emph{\(\G\)-nested} if every
collection of pairwise incomparable elements
\(S_1,\ldots,S_m\in\mathcal S\), with \(m\ge2\), has join not in \(\G\):
\[
   S_1\vee\cdots\vee S_m\notin\G.
\]

Let
\[
   N_E:=\Z^E/\Z(1,\ldots,1).
\]
For a flat \(F\), let \(v_F\in N_E\) be the image of
\(\sum_{i\in F}e_i\).  The nested-set fan of \((M,\G)\) is
\[
   \Sigma_{M,\G}
   :=
   \left\{
      \cone(v_F:F\in\mathcal S)
      \;\middle|\;
      \mathcal S\text{ is }\G\text{-nested}
   \right\}.
\]
The fan is unimodular.  Its support is the Bergman fan of \(M\), and
varying the building set gives subdivisions of the same support; see
\cite[Theorem~4.1]{FS}.  We write
\[
   X_{M,\G}:=X_{\Sigma_{M,\G}}
\]
for the associated smooth toric variety.

\subsection{The realizable de Concini--Procesi model}
\label{subsec:realizable-DCP}

Suppose that \(M\) is realized by a linear subspace
\(L\subseteq \Bbbk^E\).  Let \(T_E\subseteq \PP^{E-1}\) be the dense
projective torus and set
\[
   U_L:=\PP(L)\cap T_E.
\]
For a building set \(\G\), the de Concini--Procesi wonderful model
\(W_{L,\G}\) is the smooth compactification of \(U_L\) associated to
\(\G\); see \cite{DCP}. It can be constructed by iteratively blowing up the proper transforms of the arrangement strata indexed by flats in \(\G\setminus\{E\}\), in any order compatible with inclusion. The wonderful model \(W_{L,\G}\) sits naturally inside \(X_{M, \G}\) (see, e.g., \cite[Theorem~5.23]{Denham}).

\begin{theorem}\label{thm:wonderful-closure}
Let \(M\) be realized by \(L\subseteq\Bbbk^E\).  For every building set
\(\G\), the wonderful model \(W_{L,\G}\) is naturally isomorphic to the
closure of \(U_L\) in \(X_{M,\G}\):
\[
   W_{L,\G}\simeq \overline{U_L}^{\,X_{M,\G}}.
\]
\end{theorem}

If \(\G\subseteq\Hh\) are building sets, then the toric morphism
\(X_{M,\Hh}\to X_{M,\G}\) restricts, in the realizable case, to the usual
morphism \(W_{L,\Hh}\to W_{L,\G}\).

The rest of the paper uses the toric varieties \(X_{M,\G}\) to define Chow
rings, \(K\)-rings, pullbacks, pushforwards, and star restrictions for every
matroid.  In the realizable case, Theorem~\ref{thm:wonderful-closure} explains why
these maps are the same maps that occur on de Concini--Procesi wonderful
models.

\subsection{Feichtner--Yuzvinsky Chow rings and \(K\)-rings}
\label{subsec:FY-rings}

For a building set \(\G\), let \(x_F\) be a variable for
\(F\in\G\setminus\{E\}\). The Feichtner--Yuzvinsky Chow ring \cite{FY} is
\begin{equation}\label{eq:FY-ring}
   \A(M,\G)
   :=
   \Z[x_F:F\in\G\setminus\{E\}]/(I_\G+J_\G),
\end{equation}
where \(I_\G\) is generated by the monomials
\[
   x_{F_1}\cdots x_{F_m}
\]
for non-\(\G\)-nested collections \(F_1,\ldots,F_m\), and \(J_\G\) is
generated by the linear forms
\[
   \sum_{\substack{F\in\G\setminus\{E\}\\ i\in F}}x_F
   -
   \sum_{\substack{F\in\G\setminus\{E\}\\ j\in F}}x_F
   \qquad (i,j\in E).
\]
The ring \(\A(M, \G)\) identifies with the Chow ring of the smooth toric variety \(X_{M,\G}\).  When \(M\) is realizable, the same presentation identifies with the Chow ring of the corresponding wonderful model.

For the toric variety, \(x_F\) is the divisor class of the torus-invariant
divisor corresponding to the ray \(v_F\).  In the realizable iterated
blow-up model, \(x_F\) corresponds to the divisor created when the stratum
indexed by \(F\) is blown up, followed by taking its strict transform in
later stages.

The maximal building set is
\[
   \Gmax:=\LL(M)\setminus\{\emptyset\}.
\]
When \(\G=\Gmax\), the ring \(\A(M,\Gmax)\) is the usual Chow ring of a
matroid as in \cite[Definition~1.3]{AHK}.

The common value of the linear forms
\begin{equation}\label{eq:alpha-def}
   \alpha
   :=
   \sum_{\substack{F\in\G\setminus\{E\}\\ i\in F}}x_F
   \in\A^1(M,\G)
\end{equation}
is independent of \(i\in E\).  In realizable cases, \(\alpha\) is the pullback
of the hyperplane class from the initial projective space.

We write
\[
   \K(M,\G):=K_0(X_{M,\G})
\]
for the Grothendieck ring of vector bundles on \(X_{M,\G}\).  For a divisor
class \(D\in\A^1(M,\G)\), we write \(\OO(D)\) for the associated line bundle
and \([\OO(D)]\) for its class in \(\K(M,\G)\). When \(M\) is realizable, \(\K(M, \G)\) is also the \(K\)-ring of the corresponding wonderful model \cite[Proposition 1.13, Remark 4.4]{LLPP}. Since \(X_{M,\G}\) is smooth,
the Chern character gives an isomorphism \cite[Example 15.2.16]{Fulton}
\[
   \ch:\K(M,\G)_\Q\xrightarrow{\sim}\A(M,\G)_\Q .
\]
In addition, \cite{LLPP} constructs an integral isomorphism
\[
   \zeta_{M,\G}:\K(M,\G)\xrightarrow{\sim}\A(M,\G)
\]
and gives an explicit presentation of \(\K(M,\G)\).
\begin{lemma} \label{lem:ch-detects-K}
The abelian group \(\K(M,\G)\) is torsion-free. Consequently, the ordinary
Chern character
\[
   \ch:\K(M,\G)\to \A(M,\G)_\Q
\]
is injective. Hence two classes of \(\K(M,\G)\) are equal if and only if
they have the same Chern character.
\end{lemma}

\begin{proof}
The rational Chern character becomes an isomorphism after tensoring with
\(\Q\).  Since \(\K(M,\G)\) is torsion-free, the natural map
\(\K(M,\G)\to\K(M,\G)_\Q\) is injective.
\end{proof}

We will also use the following \(K\)-theoretic form of the
Stanley--Reisner relations.

\begin{lemma}[\(K\)-theoretic Stanley--Reisner relation]
\label{lem:K-SR}
If \(S_1,\ldots,S_m\in\G\setminus\{E\}\) are not \(\G\)-nested, then
\[
   \prod_{j=1}^m
   \left(1-[\OO(-x_{S_j})]\right)
   =0
   \qquad\text{in }\K(M,\G).
\]
\end{lemma}

\begin{proof}
This is the multiplicative Stanley--Reisner relation in the standard
presentation of the Grothendieck ring of a smooth toric variety; see, for
example, \cite[Theorem~6.4]{VezzosiVistoli}.  The rays indexed by
\(S_1,\ldots,S_m\) span a cone of \(\Sigma_{M,\G}\) if and only if
\(\{S_1,\ldots,S_m\}\) is \(\G\)-nested.  If the collection is not nested, the corresponding product vanishes. The lemma can also be seen from taking Chern character of the class.
\end{proof}

\subsection{One-step enlargements and star centers}
\label{subsec:onestep}

A \emph{one-step enlargement} is an inclusion \(\B^+=\B\cup\{F\}\), where both \(\B\) and \(\B^+\) are building sets.

Whenever this notation is used, we set
\begin{equation}\label{eq:onestep-notation}
   \Fac_\B(F)=\{F_1,\ldots,F_\ell\},
   \qquad
   \sigma_F^\B:=\cone(v_{F_1},\ldots,v_{F_\ell}).
\end{equation}
The factorization of the interval \([\emptyset,F]\) gives
\begin{equation*}
   \rk(F)=\sum_{i=1}^{\ell}\rk(F_i).
\end{equation*}

Every inclusion \(\G\subseteq\Hh\) of building sets can be connected by a
sequence
\[
   \G=\G_0\subset\G_1\subset\cdots\subset\G_N=\Hh
\]
in which each step is a one-step enlargement; see \cite[Theorem~4.2]{FM} or \cite[Proposition~5.2]{EFMPV}.

For a one-step enlargement \(\B^+=\B\cup\{F\}\), the fan
\(\Sigma_{M,\B^+}\) is the stellar subdivision of \(\Sigma_{M,\B}\) at the
cone \(\sigma_F^\B\). Write \[
   Z_F:=V(\sigma_F^\B)\subseteq X_{M,\B}, \quad \text{and } \iota_F:Z_F\hookrightarrow X_{M,\B}.
\] The map
\begin{equation*}
   \rho_F:X_{M,\B^+}\longrightarrow X_{M,\B}
\end{equation*}
is the toric blow-up with center the orbit closure \(Z_F\).

Its normal bundle is
\begin{equation*}
   N_{Z_F/X_{M,\B}}
   =
   \bigoplus_{i=1}^{\ell}\OO_{Z_F}(x_{F_i}^{\B}).
\end{equation*}
For the strict transform of the divisor \(x_{F_i}^{\B}\), we use the class
\(x_{F_i}^{\B^+}\), which satisfies
\begin{equation*}
   \rho_F^*x_{F_i}^{\B}=x_{F_i}^{\B^+}+x_F.
\end{equation*}

Following \cite[Definition~2.3]{EFMPV}, the induced building sets on the
restriction and contraction are
\[
   \B|F:=\{G\in\B\mid G\subseteq F\},
   \qquad
   \B/F:=\{(G\vee F)/F\mid G\in\B,
      \;G\not\subseteq F\}.
\]
This is a notational convention: the symbol \((M|F,\B|F)\) denotes the
product of the restricted pieces indexed by the \(\B\)-factors \(F_1,\ldots,F_\ell\) of \(F\), not
a single ordinary building set on \(M|F\):
\begin{equation*}
   \A(M|F,\B|F)
   :=
   \bigotimes_{i=1}^{\ell}\A(M|F_i,\B|F_i),
   \qquad
   \K(M|F,\B|F)
   :=
   \bigotimes_{i=1}^{\ell}\K(M|F_i,\B|F_i).
\end{equation*}
The K\"unneth formulas for Chow rings and for \(K\)-rings of these smooth
toric varieties identify the rings of the product with the tensor products
displayed above; see, for example, \cite{Totaro}. Later, any class attached to \((M|F,\B|F)\) is interpreted as the external product, or external sum in \(K\)-theory, of the corresponding classes on the factors \((M|F_i,\B|F_i)\).

The closed star of \(\sigma_F^\B\) is the restriction--contraction product
fan
\begin{equation}\label{eq:star-product}
   \Star_{\Sigma_{M,\B}}(\sigma_F^\B)
   \simeq
   \Sigma_{M|F,\B|F}\times\Sigma_{M/F,\B/F};
\end{equation}
see \cite[Lemma~5.4]{EFMPV}.  Therefore we identify
\begin{equation*}
   \A(Z_F)
   \simeq
   \A(M|F,\B|F)\otimes\A(M/F,\B/F),
   \qquad
   \K(Z_F)
   \simeq
   \K(M|F,\B|F)\otimes\K(M/F,\B/F).
\end{equation*}

In the realizable case, the same notation describes the ordinary
de Concini--Procesi center.  Namely, if \(M\) is realized by \(L\), Theorem \ref{thm:wonderful-closure} implies \(Z_F^W:=W_{L,\B}\cap Z_F\) is the blow-up center of
\[
   W_{L,\B^+}\longrightarrow W_{L,\B}.
\]
Moreover,
\[
   Z_F^W
   \simeq
   W_{L|F,\B|F}\times W_{L/F,\B/F},
\]
compatible with the star product \eqref{eq:star-product}.

\section{The tangent class of a building set}
\label{sec:Aluffi-definition}

Let \(M\) be a loopless matroid of rank \(r\), and let \(\G\) be a building
set.  The purpose of this section is to define a class
\[
   T_{M,\G}\in\K(M,\G)
\]
which behaves as the tangent bundle of a de Concini--Procesi wonderful model
when \(M\) is realizable.  The construction is modeled on Aluffi's formula
for the Chern class of a blow-up along a complete intersection.

\subsection{Aluffi's complete-intersection blow-up formula}

\begin{lemma}[{\cite[Lemma~1.3]{AluffiBlowups}}]
\label{lem:Aluffi}
Let \(X\) be a smooth variety, and let \(Z\subset X\) be a smooth complete
intersection of smooth divisors \(D_1,\ldots,D_\ell\) meeting transversely.
Let
\[
   \rho:\widetilde X=\Bl_ZX\to X
\]
be the blow-up, let \(E_{\mathrm{exc}}\) be the exceptional divisor, and let
\(\widetilde D_i\) be the strict transform of \(D_i\).  Then
\[
   c(T_{\widetilde X})
   =
   \rho^*c(T_X)(1+E_{\mathrm{exc}})
   \prod_{i=1}^{\ell}
   \frac{1+\rho^*D_i-E_{\mathrm{exc}}}{1+\rho^*D_i}.
\]
Equivalently, since \(\rho^*D_i=\widetilde D_i+E_{\mathrm{exc}}\),
\begin{equation}\label{eq:Aluffi-Chern}
   c(T_{\widetilde X})
   =
   \rho^*c(T_X)(1+E_{\mathrm{exc}})
   \prod_{i=1}^{\ell}
   \frac{1+\widetilde D_i}{1+\widetilde D_i+E_{\mathrm{exc}}}.
\end{equation}
\end{lemma}

\begin{definition}[Aluffi \(K\)-recursion]
\label{def:Aluffi-K-recursion}
In the situation of Lemma~\ref{lem:Aluffi}, set
\begin{equation}\label{eq:Aluffi-K-correction}
   R_\rho
   :=
   [\OO(E_{\mathrm{exc}})]-[\OO]
   +
   \sum_{i=1}^{\ell}
   \left(
      [\OO(\widetilde D_i)]-[\OO(\widetilde D_i+E_{\mathrm{exc}})]
   \right)
   \in K_0(\widetilde X).
\end{equation}
We say that classes \(\T_X\in K_0(X)\) and
\(\T_{\widetilde X}\in K_0(\widetilde X)\) satisfy the \emph{Aluffi
\(K\)-recursion} for \(\rho\) if
\begin{equation}\label{eq:Aluffi-K-recursion}
   \T_{\widetilde X}=\rho^*\T_X+R_\rho.
\end{equation}
Taking total Chern classes in \eqref{eq:Aluffi-K-recursion} gives the
strict-transform form \eqref{eq:Aluffi-Chern}.
\end{definition}

\subsection{Cutting divisors and the intrinsic tangent class}

For \(F\in\G\setminus\{E\}\), define the divisor class
\begin{equation}\label{eq:theta-def}
   \theta_F^\G
   :=
   \alpha-
   \sum_{\substack{H\in\G\\F\subseteq H\subsetneq E}}x_H
   \in\A^1(M,\G).
\end{equation}
We call \(\theta_F^\G\) the \emph{cutting class} for \(F\).

The notation is motivated by the realizable case.  Suppose that \(M\) is
realized by \(L\).  In the wonderful blow-up construction, the center
corresponding to \(F\) is cut out, before the blow-up of that center, by
\(\rk(F)\) general hyperplanes containing the linear space indexed by \(F\).
After the blow-ups of the larger flats \(H\supseteq F\), the strict transform
of each such hyperplane has class \(\theta_F^\G\).  Therefore the blow-up at
\(F\) contributes the Chern-class factor
\[
   (1+x_F)
   \left(
      \frac{1+\theta_F^\G}{1+\theta_F^\G+x_F}
   \right)^{\rk(F)}.
\]
When \(\rk(F)=1\), the definition of \(\alpha\) gives
\(\theta_F^\G=0\), and this factor is equal to \(1\), as expected for a
divisorial blow-up center.

\begin{definition}[Tangent class]
\label{def:tangent}
The \emph{tangent class} of \((M,\G)\) is the \(K\)-class
\begin{equation}\label{eq:intrinsic-T-def}
\begin{aligned}
   T_{M,\G}
   :=&\; r[\OO(\alpha)]-[\OO]  \\
   &+
   \sum_{F\in\G\setminus\{E\}}
   \left(
      [\OO(x_F)]-[\OO]
      +
      \rk(F)
      \left(
         [\OO(\theta_F^\G)]
         -
         [\OO(\theta_F^\G+x_F)]
      \right)
   \right)
   \in\K(M,\G).
\end{aligned}
\end{equation}
\end{definition}

Its total Chern class is
\begin{equation}\label{eq:intrinsic-cT}
   c(T_{M,\G})
   =
   (1+\alpha)^r
   \prod_{F\in\G\setminus\{E\}}
   (1+x_F)
   \left(
      \frac{1+\theta_F^\G}{1+\theta_F^\G+x_F}
   \right)^{\rk(F)}.
\end{equation}
Thus, when \(M\) is realizable, Definition~\ref{def:tangent} recovers the
ordinary tangent bundle of the wonderful model. It is useful to separate the logarithmic part from the boundary-normal part.
Set
\[
\begin{aligned}
   T^{\log}_{M,\G}
   :=
   &\; r[\OO(\alpha)]-[\OO]  \\
   &+
   \sum_{F\in\G\setminus\{E\}}
      \rk(F)
      \left(
         [\OO(\theta_F^\G)]
         -
         [\OO(\theta_F^\G+x_F)]
      \right).
\end{aligned}
\]
Then
\[
   T_{M,\G}
   =
   T^{\log}_{M,\G}
   +
   \sum_{F\in\G\setminus\{E\}}
      \bigl([\OO(x_F)]-[\OO]\bigr).
\]
In the realizable case, let
\[
   D=\sum_{F\in\G\setminus\{E\}}D_F
\]
be the boundary divisor of \(W_{L,\G}\). Then \(T^{\log}_{M,\G}\) is the
class of the logarithmic tangent bundle
\[
   T_{W_{L,\G}}(-\log D).
\]
Indeed, the exact sequence
\[
   0\to T_{W_{L,\G}}(-\log D)\to T_{W_{L,\G}}
      \to \bigoplus_{F\in\G\setminus\{E\}}\OO_{D_F}(D_F)\to 0
\]
and the identity
\[
   [\OO_{D_F}(D_F)]=[\OO(D_F)]-[\OO]
\]
show that the boundary terms in Definition~\ref{def:tangent} are precisely
the normal contributions from the boundary components. Thus the Euler term and the hyperplane-correction terms form the formal
logarithmic tangent class, while the terms
\[
   [\OO(x_F)]-[\OO]
\]
are the boundary-normal contributions.

\begin{example}[Five points and ten lines in \(\PP^3\)]
\label{ex:P3-points-lines}
Consider the braid matroid \(K_5\) with minimal building set.  In the
realizable model, one starts from \(\PP^3\), blows up five points \(p_i\),
and then blows up the strict transforms of the ten lines \(L_{ij}\) through
pairs of points.  Let \(h\) be the hyperplane class, and let \(e_i\) and
\(f_{ij}\) be the corresponding exceptional divisor classes.

For a point \(p_i\), the cutting class is
\[
   \theta_{p_i}=h-e_i.
\]
For a line \(L_{ij}\), after the point blow-ups and before the line blow-up,
the strict transform of a hyperplane containing \(L_{ij}\) has class
\(h-e_i-e_j\); after the line blow-up the cutting class is
\[
   \theta_{L_{ij}}=h-e_i-e_j-f_{ij}.
\]
Formula~\eqref{eq:intrinsic-cT} gives
\begin{equation}\label{eq:P3-example-cT}
\begin{aligned}
   c(T)
   =&\;(1+h)^4
   \prod_{i=1}^{5}
   (1+e_i)
   \left(\frac{1+h-e_i}{1+h}\right)^3   \\
   &\times
   \prod_{1\le i<j\le5}
   (1+f_{ij})
   \left(
      \frac{1+h-e_i-e_j-f_{ij}}{1+h-e_i-e_j}
   \right)^2 .
\end{aligned}
\end{equation}
\end{example}
\begin{proposition}[Maximal building set]
\label{prop:maximal-building-set}
Let
\[
   \Gmax=\LL(M)\setminus\{\emptyset\}.
\]
Then \(T_{M,\Gmax}\) agrees with the tangent class \(T_M\) of
\cite[Theorem~3.4]{ChengTangent}.  More explicitly,
\[
   T_{M,\Gmax}
   =
   -[\OO]
   +
   \sum_{F\subsetneq E}
      \bigl([\OO(x_F)]-[\OO]\bigr)
   +
   \sum_{k=1}^{r}
      \left[
        \OO\!\left(
           \alpha-
           \sum_{\substack{H\subsetneq E\\ r>\rk(H)\ge k}}x_H
        \right)
      \right].
\]
\end{proposition}

\begin{proof}
For \(F\in\Gmax\setminus\{E\}\), set
\[
   L_{\rk(F)}
   :=
   \alpha-
   \sum_{\substack{H\subsetneq E\\r>\rk(H)\ge \rk(F)}}x_H .
\]
We first claim that the \(F\)-summand in the hyperplane-correction part may
be written as
\[
   \rk(F)
   \left(
      [\OO(L_{\rk(F)})]
      -
      [\OO(L_{\rk(F)}+x_F)]
   \right).
\]
Indeed,
\[
   \theta_F^{\Gmax}
   =
   \alpha-
   \sum_{\substack{H\supseteq F\\H\subsetneq E}}x_H.
\]
The difference between \(L_{\rk(F)}\) and \(\theta_F^{\Gmax}\) is a sum of
classes \(x_H\) with \(\rk(H)\ge \rk(F)\) and \(H\nsupseteq F\).  Such an
\(H\) is incomparable with \(F\), by Lemma~\ref{lem:K-SR},
\[
   (1-[\OO(-x_F)])(1-[\OO(-x_H)])=0.
\]
This relation implies that subtracting such an \(x_H\) from both line bundles
does not change the difference \([\OO(D)]-[\OO(D+x_F)]\). Applying this to all such \(H\) gives the claim.

Therefore
\[
\begin{aligned}
   T_{M,\Gmax}
   =
   &\;r[\OO(\alpha)]-[\OO]
   +
   \sum_{F\subsetneq E}
      \bigl([\OO(x_F)]-[\OO]\bigr)  \\
   &+
   \sum_{F\subsetneq E}
      \rk(F)
      \left(
         [\OO(L_{\rk(F)})]
         -
         [\OO(L_{\rk(F)}+x_F)]
      \right).
\end{aligned}
\]
For \(1\le k\le r-1\), put
\[
   L_k:=
   \alpha-
   \sum_{\substack{H\subsetneq E\\r>\rk(H)\ge k}}x_H.
\]
Since flats of the same rank are pairwise incomparable, another application
of Lemma~\ref{lem:K-SR} gives the telescoping identity
\[
   \sum_{\rk(F)=k}
      \left(
         [\OO(L_k)]-[\OO(L_k+x_F)]
      \right)
   =
   [\OO(L_k)]-[\OO(L_{k+1})].
\]
Thus the hyperplane-correction part is
\[
   \sum_{k=1}^{r-1}k\bigl([\OO(L_k)]-[\OO(L_{k+1})]\bigr).
\]
Since \(\alpha = L_r\), telesoping again yields the displayed formula.
\end{proof}
\begin{remark}
The same argument applies to any building set \(\G\) with the property that
every incomparable pair \(F,G\in\G\setminus\{E\}\) is non-nested, equivalently
\(F\vee G\in\G\). This corresponds to the case of polymatroids \cite{CHLSW}.
\end{remark}

\subsection{Compatibility with one-step building-set blow-ups}
\label{subsec:relative-form}

Let \(\B^+=\B\cup\{F\}\) be a one-step enlargement, and write
\[
   \Fac_\B(F)=\{F_1,\ldots,F_\ell\}.
\]
Let
\[
   \rho_F:X_{M,\B^+}\to X_{M,\B}
\]
be the toric morphism induced by the stellar subdivision at
\(\sigma_F^\B=\cone(v_{F_1},\ldots,v_{F_\ell})\).

\begin{lemma}[Pullback of boundary divisors]
\label{lem:pullback-boundary-divisors}
For every \(G\in\B\setminus\{E\}\),
\[
   \rho_F^*x_G^\B
   =
   \begin{cases}
      x_G^{\B^+}+x_F, & G\in\Fac_\B(F),\\[3pt]
      x_G^{\B^+}, & G\notin\Fac_\B(F).
   \end{cases}
\]
\end{lemma}

\begin{proof}
The exceptional divisor of \(\rho_F\) is the divisor \(D_F\) corresponding to the new ray \(v_F\). The center is
\[
   V(\sigma_F^\B)=D_{F_1}^\B\cap\cdots\cap D_{F_\ell}^\B.
\]
Thus \(D_G^\B\) contains the center exactly for
\(G\in\{F_1,\ldots,F_\ell\}\), and the multiplicity is one.  The strict
transform of \(D_G^\B\) is \(D_G^{\B^+}\), giving the stated formula.
\end{proof}

\begin{lemma}[Factor hyperplane cancellation]
\label{lem:factor-hyperplane-cancellation}
Set
\[
   A_i:=\theta_{F_i}^{\B^+}+x_{F_i}^{\B^+},
   \qquad
   \Theta:=\theta_F^{\B^+}.
\]
Then, in \(\K(M,\B^+)\),
\begin{equation}\label{eq:factor-K-cancellation}
   [\OO(A_i+x_F)]-[\OO(A_i)]
   =
   [\OO(\Theta+x_F)]-[\OO(\Theta)]
\end{equation}
for every \(i\).
\end{lemma}

\begin{proof}
By the definition of the cutting classes,
\[
   A_i-\Theta
   =
   -
   \sum_{H\in\mathcal S_i}x_H,
   \qquad
   \mathcal S_i:=
   \left\{
      H\in\B^+
      \;\middle|\;
      F_i\subsetneq H,
      \ F\nsubseteq H
   \right\}.
\]
If \(H\in\mathcal S_i\), then \(H\wedge F=F_i\neq\emptyset\). Since two elements of a building set with nonzero meet have their join in the
building set, \(H\vee F\in\B^+\), and \(\{H,F\}\) is not \(\B^+\)-nested.
By Lemma~\ref{lem:K-SR},
\[
   \left(1-[\OO(-x_H)]\right)
   \left(1-[\OO(-x_F)]\right)=0.
\]
Put \(m_H=[\OO(-x_H)]\), the preceding relation
says \(m_H(1-[\OO(-x_F)])=1-[\OO(-x_F)]\), and hence
\[
   \prod_{H\in\mathcal S_i}m_H\,(1-[\OO(-x_F)])=1-[\OO(-x_F)].
\]
We get
\[
\begin{aligned}
   [\OO(A_i+x_F)]-[\OO(A_i)]
   &=
   [\OO(A_i)]\bigl([\OO(x_F)]-1\bigr) \\
   &=
   [\OO(\Theta)]\prod_{H\in\mathcal S_i}m_H\bigl([\OO(x_F)]-1\bigr) \\
   &=
   [\OO(\Theta)]\bigl([\OO(x_F)]-1\bigr) \\
   &=
   [\OO(\Theta+x_F)]-[\OO(\Theta)].
\end{aligned}
\] \end{proof}

\begin{proposition}[Relative Aluffi identity]
\label{prop:relative-Aluffi}
For a one-step enlargement \(\B^+=\B\cup\{F\}\), the tangent classes satisfy
\begin{equation}\label{eq:relative-Aluffi-K}
   T_{M,\B^+}
   =
   \rho_F^*T_{M,\B}
   +[\OO(x_F)]-[\OO]
   +
   \sum_{i=1}^{\ell}
   \left(
      [\OO(x_{F_i}^{\B^+})]
      -[\OO(x_{F_i}^{\B^+}+x_F)]
   \right).
\end{equation}
Consequently,
\begin{equation}\label{eq:relative-Aluffi-Chern}
   c(T_{M,\B^+})
   =
   \rho_F^*c(T_{M,\B})(1+x_F)
   \prod_{i=1}^{\ell}
   \frac{1+x_{F_i}^{\B^+}}{1+x_{F_i}^{\B^+}+x_F}.
\end{equation}
\end{proposition}

\begin{proof}
For \(K\in\B\setminus\{E\}\), the cutting classes satisfy
\begin{equation}\label{eq:theta-pullback-onestep}
   \rho_F^*\theta_K^\B=\theta_K^{\B^+}.
\end{equation}
Indeed, \(\rho_F^*\alpha=\alpha\).  If \(K\nsubseteq F\), the new flat
\(F\) does not occur in the defining sum for \(\theta_K^{\B^+}\).  If
\(K\subseteq F\), then \(K\) is contained in a unique \(\B\)-factor
\(F_i\) of \(F\), and the exceptional term in
\(\rho_F^*x_{F_i}^\B=x_{F_i}^{\B^+}+x_F\) is exactly the new \(x_F\)-term in
\(\theta_K^{\B^+}\).

We compare \(T_{M,\B^+}\) with \(\rho_F^*T_{M,\B}\).  All summands indexed
by flats \(K\in\B\setminus\{E,F_1,\ldots,F_\ell\}\) match by
\eqref{eq:theta-pullback-onestep} and Lemma~\ref{lem:pullback-boundary-divisors}.
The boundary terms for the factors contribute
\[
   \sum_{i=1}^{\ell}
   \left(
      [\OO(x_{F_i}^{\B^+})]
      -[\OO(x_{F_i}^{\B^+}+x_F)]
   \right).
\]
The new flat \(F\) contributes
\[
   [\OO(x_F)]-[\OO]
   +
   \rk(F)
   \left([
      \OO(\Theta)]-[\OO(\Theta+x_F)]
   \right),
\]
where \(\Theta=\theta_F^{\B^+}\).  Finally, the hyperplane-correction terms
for the old factors \(F_i\) contribute
\[
   \sum_{i=1}^{\ell}
   \rk(F_i)
   \left(
      [\OO(A_i+x_F)]-[\OO(A_i)]
   \right),
   \quad \text{where }
   A_i=\theta_{F_i}^{\B^+}+x_{F_i}^{\B^+}.
\]
By Lemma~\ref{lem:factor-hyperplane-cancellation}, the last display is
\[
   \sum_{i=1}^{\ell}\rk(F_i)
   \left(
      [\OO(\Theta+x_F)]-[\OO(\Theta)]
   \right).
\]
Since \(\rk(F)=\sum_i\rk(F_i)\), this cancels the hyperplane-correction part
of the new \(F\)-summand.  The remaining terms are exactly
\eqref{eq:relative-Aluffi-K}. Taking total Chern classes gives
\eqref{eq:relative-Aluffi-Chern}.
\end{proof}

\begin{remark}[The Boolean ambient model and the normal class]
\label{rem:boolean-ambient-normal}
Let \(\G\subseteq L(M)\setminus\{\emptyset\}\) be a building set
containing the top flat \(E\).  Define
\[
  \mathcal U_\G
  :=
  \{\,\varnothing\neq S\subseteq E
      \mid \operatorname{cl}_M(S)\in \G\,\}.
\]
Then \(\mathcal U_\G\) is a top-containing building set in the Boolean
lattice \(2^E\). This is because for the Boolean lattice, a collection \(\mathcal B \subseteq 2^E \setminus \{\varnothing\}\) is a building set iff
\begin{enumerate}
    \item \(\B\) contains all singleton, and
    \item \(S, T \in \B,\, S \cap T \neq \varnothing \Rightarrow S \cup T \in \B\). 
\end{enumerate}

Let \(X_{E,\mathcal U_\G}\) be the smooth projective toric variety
associated to the Boolean building set \(\mathcal U_\G\).  Its total
Chern class is
\[
  c(T_{X_{E,\mathcal U_\G}})
  =
  \prod_{S\in \mathcal U_\G\setminus\{E\}} (1+y_S),
\]
where \(y_S\) is the torus-invariant divisor class indexed by \(S\).

Assume now that \(M\) is realized by a linear subspace \(L\).  The
\(\G\)-nested fan of \(M\) is the subfan of the
\(\mathcal U_\G\)-nested fan whose rays are indexed by flats.  Hence
there are natural inclusions
\[
  W_{L,\G}\subseteq X_{M,\G}\subseteq X_{E,\mathcal U_\G}.
\]
Write
\[
  i_\G\colon W_{L,\G}\hookrightarrow X_{E,\mathcal U_\G}
\]
for the resulting embedding, and set the normal bundle
\[
  \mathcal N_{L,\G}
  :=
  N_{W_{L,\G}/X_{E,\mathcal U_\G}}.
\]

The toric divisors indexed by non-flats are disjoint from
\(X_{M,\G}\), hence from \(W_{L,\G}\).  The divisor indexed by a flat
\(F\in \G\setminus\{E\}\) restricts to the boundary divisor class
\(x_F\).  Therefore
\[
  c(i_\G^*T_{X_{E,\mathcal U_\G}})
  =
  \prod_{F\in \G\setminus\{E\}}(1+x_F).
\]
The normal exact sequence
\[
  0\longrightarrow T_{W_{L,\G}}
  \longrightarrow i_\G^*T_{X_{E,\mathcal U_\G}}
  \longrightarrow \mathcal N_{L,\G}
  \longrightarrow 0
\]
gives
\[
  [\mathcal N_{L,\G}]
  =
  \sum_{F\in \G\setminus\{E\}}
    \bigl([\mathcal O(x_F)]-[\mathcal O]\bigr)
  +
  (|E|-1)[\mathcal O]
  -
  [T_{W_{L,\G}}]
\]
in \(K\)-theory. Motivated by this formula, for an arbitrary loopless matroid \(M\) we
define the normal class
\[
  N_{M,\G}
  :=
  \sum_{F\in \G\setminus\{E\}}
    \bigl([\mathcal O(x_F)]-[\mathcal O]\bigr)
  +
  (|E|-1)[\mathcal O]
  -
  T_{M,\G}
  \in K(M,\G).
\]
When \(M\) is realized by \(L\), this class is represented by the actual
normal bundle \(\mathcal N_{L,\G}\).

For a one-step enlargement \(\B^+=\B\cup\{F\}\), Proposition~\ref{prop:relative-Aluffi} is equivalent to the pullback identity
\[
  N_{M,\B^+}=\rho_F^*N_{M,\B}.
\]

It is natural to ask whether the class \(N_{M,\G}\) is the restriction
of a canonical Berget--Eur--Spink--Tseng type quotient class \cite{BEST23}
\[
  \mathscr Q_{M,\G}\in K_0(X_{E,\mathcal U_\G})
\]
of rank \(|E|-\operatorname{rk}(M)\).  In the realizable case, such a
class should restrict to \(\mathcal N_{L,\G}\), and one would expect
\(W_{L,\G}\) to arise as the zero locus of a regular section, extending
the quotient-bundle picture to arbitrary building sets.
\end{remark}
\section{The formal HRR package and its tangent realization}
\label{sec:formal-HRR-and-realization}

In the previous section we defined the tangent class
\[
   T_{M,\G}\in\K(M,\G)
\]
and proved its compatibility with one-step building-set blow-ups.  In this
section we first define the formal Hirzebruch--Riemann--Roch package for
arbitrary building sets.  The main theorem then proves that this package is
realized by the tangent class \(T_{M,\G}\).

For a \(K\)-class \(\xi\), set
\[
   \lambda_y(\xi):=\sum_{p\ge0}[\wedge^p\xi]y^p
\]
using the \(\lambda\)-ring structure on \(K\)-theory, and define
\begin{equation}\label{eq:Phi-y-def}
   \Phi_y(\xi)
   :=
   \ch(\lambda_y\xi^\vee)\td(\xi)
   \in \A(-)_\Q[y].
\end{equation}

Equivalently, if the Chern roots of \(\xi\) are \(u_1,\ldots,u_d\), then
\[
   \Phi_y(\xi)=\prod_{a=1}^d Q_y(u_a),
   \qquad
   Q_y(u):=\frac{u(1+ye^{-u})}{1-e^{-u}},
\] and \(\Phi_y\) is multiplicative for direct sums in \(K\)-theory.

If \(\xi=T_X\) is the tangent bundle of a smooth variety, then
\[
   \Phi_y(T_X)=\ch(\lambda_y\Omega_X^1)\td(T_X)
\] is the cohomological Hirzebruch class in the smooth case \cite{BSY}.

\subsection{The realizable case}
\label{subsec:realizable-guide-HRR}

Suppose first that \(M\) is realized by a linear subspace \(L\), and write
\[
   W_\G:=W_{L,\G}
\]
for the de Concini--Procesi wonderful model.  If \(\G\subseteq\Hh\), the
morphism
\[
   \rho:W_\Hh\longrightarrow W_\G
\]
is an iterated blow-up along smooth centers.  Hence
\[
   R^i\rho_*\OO_{W_\Hh}=0\quad(i>0),
   \qquad
   \rho_*\OO_{W_\Hh}=\OO_{W_\G}.
\]
Grothendieck--Riemann--Roch for \(\rho\) and \(\OO_{W_\Hh}\) gives
\[
   \rho_*\td(T_{W_\Hh})=\td(T_{W_\G}).
\]
This motivates defining the formal Todd class by pushforward from the
maximal building set.

The corresponding Hirzebruch class is
\[
   \mathfrak H_y(W_\G)
   :=
   \Phi_y(T_{W_\G})
   =
   \ch(\lambda_yT_{W_\G}^\vee)\td(T_{W_\G}).
\]
If \(M\) is moreover realizable over \(\mathbb C\), the wonderful model is obtained from projective space by iterated blow-ups along smooth centers whose cohomology is generated by algebraic classes. By the blow-up formula, the same is true for \(W_\G\), and hence its cohomology is of type \((p,p)\). Thus \(H^{2p}(W_\G,\Q)\) is generated by algebraic classes and
\(H^{p,q}(W_\G)=0\) for \(p\ne q\). Therefore
\[
   \dim_\Q\A^p(W_\G)=h^{p,p}(W_\G)=(-1)^p\chi(W_\G,\Omega_{W_\G}^p).
\]
Hirzebruch--Riemann--Roch therefore gives
\[
   \deg_{W_\G}\bigl(\mathfrak H_y(W_\G)|_{y=-t}\bigr)
   =
   \sum_{p\ge0}\dim_\Q\A^p(W_\G)t^p.
\]
The rest of the section reproduces this package in the Feichtner--Yuzvinsky
rings for arbitrary matroids and building sets.

\subsection{The formal Todd class and HRR identity}

Let \(\Gmax\) be the maximal building set of \(M\).  For any building set
\(\G\), let
\[
   \rho_{\Gmax,\G}:X_{M,\Gmax}\longrightarrow X_{M,\G}
\]
be the toric morphism induced by the refinement of nested-set fans.  Define
the \emph{formal Todd class of \((M,\G)\)} by
\begin{equation}\label{eq:formal-Todd-def}
   \Td_{M,\G}
   :=
   (\rho_{\Gmax,\G})_*\td(T_{M,\Gmax})
   \in \A(M,\G)_\Q.
\end{equation}

\begin{definition}\label{def:chi-MG}
Following \cite[Remark~4.4]{LLPP}, define the Euler characteristic on
\(\K(M,\G)\) by pullback to the maximal building set:
\[
   \chi_{M,\G}(\xi)
   :=
   \chi_{M,\Gmax}\bigl((\rho_{\Gmax,\G})^*\xi\bigr).
\]
\end{definition}

\begin{proposition}[Formal HRR]
\label{prop:formal-HRR}
For every \(\xi\in\K(M,\G)\),
\[
   \chi_{M,\G}(\xi)
   =
   \int_{M,\G}\ch(\xi)\Td_{M,\G}.
\]
\end{proposition}

\begin{proof}
By definition of \(\chi_{M,\G}\) and the HRR identity on the maximal
building set,
\[
   \chi_{M,\G}(\xi)
   =
   \int_{M,\Gmax}
      \ch((\rho_{\Gmax,\G})^*\xi)\td(T_{M,\Gmax}).
\]
The Chern character commutes with pullback, and the projection formula gives
\[
   \chi_{M,\G}(\xi)
   =
   \int_{M,\G}
      \ch(\xi)(\rho_{\Gmax,\G})_*\td(T_{M,\Gmax})
   =
   \int_{M,\G}\ch(\xi)\Td_{M,\G}.
\]
\end{proof}

\subsection{The recursive Hirzebruch class}

We now define the formal Hirzebruch class
\[
   \Hy(M,\G)\in\A(M,\G)_\Q[y].
\]
For the maximal building set, set
\begin{equation}\label{eq:Hy-max-def}
   \Hy(M,\Gmax):=\Phi_y(T_{M,\Gmax}).
\end{equation}

Suppose that \(\B^+=\B\cup\{F\}\) is a one-step enlargement.  We use the
standing notation of Section~\ref{subsec:onestep}:
\[
   \Fac_\B(F)=\{F_1,\ldots,F_\ell\},
   \qquad
   \rho_F:X_{M,\B^+}\to X_{M,\B},
   \qquad
   \iota_F:Z_F\hookrightarrow X_{M,\B}.
\]
Under the star-fan identification,
\[
   \A(Z_F)
   \simeq
   \A(M|F,\B|F)\otimes\A(M/F,\B/F),
\]
and similarly in \(K\)-theory.  External products are denoted by
\(\boxtimes\).  Define
\[
   q_\ell(y):=(-y)+(-y)^2+\cdots+(-y)^{\ell-1}.
\]
For a non-maximal building set \(\B\), define \(\Hy(M,\B)\) recursively by
\begin{equation}\label{eq:Hy-recursion}
\begin{aligned}
   \Hy(M,\B)
   :=&\; (\rho_F)_*\Hy(M,\B^+)        \\
   &-
   q_\ell(y)(\iota_F)_*
   \left(
      \Hy(M|F,\B|F)\boxtimes\Hy(M/F,\B/F)
   \right).
\end{aligned}
\end{equation}
To make this a construction, choose a chain
\[
   \G=\G_0\subset\G_1\subset\cdots\subset\G_N=\Gmax
\]
in which each step adds one flat.  Such chains exist by the one-step
enlargement theorem recalled in Section~\ref{subsec:onestep}. The terms involving \(M|F\) and \(M/F\) are already defined by induction on
rank, using the product convention for the restriction factor. A priori the
resulting class depends on the chosen chain.  Theorem~\ref{thm:tangent-realization}
below implies that it does not. 

\begin{proposition}[Formal Chow-polynomial formula]
\label{prop:formal-Chow-polynomial}
For every building set \(\G\),
\[
   \deg_{M,\G}\bigl(\Hy(M,\G)|_{y=-t}\bigr)
   =
   H_M^\G(t)
   :=
   \sum_{p\ge0}\dim_\Q\A^p(M,\G)t^p.
\]
Equivalently, if \(\Hy(M,\G)=\sum_p\Theta^p_{M,\G}y^p\), then
\[
   \dim_\Q\A^p(M,\G)=(-1)^p\deg_{M,\G}\Theta^p_{M,\G}.
\]
\end{proposition}

\begin{proof}
For \(\Gmax\), this is the maximal-building-set formula from
\cite{ChengTangent}.  Suppose \(\B^+=\B\cup\{F\}\) and
\(\ell=|\Fac_\B(F)|\).  Evaluating \eqref{eq:Hy-recursion} at \(y=-t\) and
taking degrees gives
\[
\begin{aligned}
   \deg_{M,\B}\bigl(\Hy(M,\B)|_{y=-t}\bigr)
   =&\;
   \deg_{M,\B^+}\bigl(\Hy(M,\B^+)|_{y=-t}\bigr) \\
   &-
   (t+\cdots+t^{\ell-1})
   \deg_{Z_F}
   \left(
      \Hy(M|F,\B|F)|_{y=-t}\boxtimes
      \Hy(M/F,\B/F)|_{y=-t}
   \right).
\end{aligned}
\]
By induction on rank and on the length of the chosen chain, this equals
\[
   H_M^{\B^+}(t)
   -
   (t+\cdots+t^{\ell-1})
   H_{M|F}^{\B|F}(t)H_{M/F}^{\B/F}(t).
\]
The one-step Hilbert-series formula \cite[Theorem~5.1]{EFMPV} says
\[
   H_M^{\B^+}(t)
   =
   H_M^\B(t)
   +
   (t+\cdots+t^{\ell-1})
   H_{M|F}^{\B|F}(t)H_{M/F}^{\B/F}(t),
\]
so the displayed degree is \(H_M^\B(t)\).
\end{proof}

The preceding construction is formal.  The main point is the following theorem.

\begin{theorem}[Tangent realization of the HRR package]
\label{thm:tangent-realization}
For every loopless matroid \(M\) and every building set \(\G\),
\[
   \Hy(M,\G)=\Phi_y(T_{M,\G}).
\]
In particular,
\[
   \Td_{M,\G}=\td(T_{M,\G}),
\]
and
\[
   \deg_{M,\G}
   \left(
      \ch(\lambda_yT_{M,\G}^\vee)\td(T_{M,\G})
   \right)\bigg|_{y=-t}
   =
   H_M^\G(t).
\]
\end{theorem}

The rest of the section is devoted to the proof.

\subsection{A universal one-step identity}

\begin{lemma}[One-step Hirzebruch identity]
\label{lem:universal-one-step-HRR}
Let
\[
   \rho:\widetilde X=\Bl_ZX\to X
\]
be the blow-up of a smooth variety \(X\) along a smooth complete intersection
\[
   Z=D_1\cap\cdots\cap D_\ell
\]
of smooth divisors meeting transversely.  Let \(E_{\rm exc}\) be the
exceptional divisor, let \(\widetilde D_i\) be the strict transform of
\(D_i\), and let
\[
   N=\bigoplus_{i=1}^{\ell}\OO_Z(D_i)
\]
be the normal bundle of \(Z\) in \(X\).  Suppose classes
\[
   \T_X\in K_0(X),
   \qquad
   \T_{\widetilde X}\in K_0(\widetilde X),
   \qquad
   \T_Z\in K_0(Z)
\]
satisfy
\begin{align*}
   \T_{\widetilde X}
   &=
   \rho^*\T_X
   +[\OO(E_{\rm exc})]-[\OO]
   +
   \sum_{i=1}^{\ell}
   \left(
      [\OO(\widetilde D_i)]-[\OO(\widetilde D_i+E_{\rm exc})]
   \right), \\
   \T_Z&=\iota^*\T_X-N,
   \qquad
   \iota:Z\hookrightarrow X.
\end{align*}
Then
\begin{equation}\label{eq:universal-one-step-HRR}
   \rho_*\Phi_y(\T_{\widetilde X})
   =
   \Phi_y(\T_X)+q_\ell(y)\,\iota_*\Phi_y(\T_Z),
\end{equation}
where \(q_\ell(y)=(-y)+(-y)^2+\cdots+(-y)^{\ell-1}\).
\end{lemma}

\begin{proof}
Set
\[
   R_\rho
   :=
   [\OO(E_{\rm exc})]-[\OO]
   +
   \sum_{i=1}^{\ell}
   \left(
      [\OO(\widetilde D_i)]-[\OO(\widetilde D_i+E_{\rm exc})]
   \right).
\]
Since \(\Phi_y\) is multiplicative for direct sums in \(K\)-theory,
\[
   \Phi_y(\T_{\widetilde X})
   =
   \rho^*\Phi_y(\T_X)\Phi_y(R_\rho),
   \qquad
   \Phi_y(\T_Z)
   =
   \iota^*\Phi_y(\T_X)\Phi_y(-N).
\]
By the projection formula, it is enough to prove
\begin{equation}\label{eq:universal-reduced-R}
   \rho_*\Phi_y(R_\rho)
   =
   1+q_\ell(y)\,\iota_*\Phi_y(-N), 
\end{equation} which is independent of \(\T_X\).

Now take \(\T_X=T_X\), \(\T_{\widetilde X}=T_{\widetilde X}\), and
\(\T_Z=T_Z\).  By \cite[Example~3.3]{BSY},
\[
   \rho_*\Phi_y(T_{\widetilde X})
   =
   \Phi_y(T_X)+q_\ell(y)\,\iota_*\Phi_y(T_Z).
\]
Aluffi's formula says that \(\Phi_y(T_{\widetilde X})=\Phi_y(\rho^*T_X) \Phi_y(R_\rho)\), and the tangent-normal exact sequence gives
\[
   T_Z=\iota^*T_X-N.
\]
Substituting these two identities and applying the projection formula shows
\[
   \rho_*\Phi_y(R_\rho)
   =
   1+q_\ell(y)\,\iota_*\Phi_y(-N).
\]
This proves \eqref{eq:universal-reduced-R}, and the general identity follows.
\end{proof}

\subsection{The star-normal identity}

Proposition~\ref{prop:relative-Aluffi} gives the first hypothesis needed in
Lemma~\ref{lem:universal-one-step-HRR}.  We now verify the second hypothesis,
namely the center identity.  Let \(\B^+=\B\cup\{F\}\) be a one-step
enlargement, and write
\[
   \Fac_\B(F)=\{F_1,\ldots,F_\ell\}.
\]
Recall that \(Z_F\) is the closed-star stratum of
\(\sigma_F^\B=\cone(v_{F_1},\ldots,v_{F_\ell})\).  By the star-fan product,
\[
   \A(Z_F)
   \simeq
   \A(M|F,\B|F)\otimes\A(M/F,\B/F),
\]
and similarly in \(K\)-theory.  We use the product convention for the
restriction factor from Section~\ref{subsec:onestep}.

We use left/right notation.  If \(G\subsetneq F_i\), set
\[
   x_G^L:=x_G^{\B|F_i}\boxtimes 1,
   \qquad
   \theta_G^L:=\theta_G^{\B|F_i}\boxtimes 1.
\]
Here \(\theta_G^{\B|F_i}\) is the cutting class for the smaller pair
\((M|F_i,\B|F_i)\).  If \(F\subsetneq G\subsetneq E\), set
\[
   x_G^R:=1\boxtimes x_{G/F}^{\B/F},
   \qquad
   \theta_G^R:=1\boxtimes\theta_{G/F}^{\B/F},
   \qquad
   \alpha^R:=1\boxtimes\alpha_{M/F}.
\]
For each \(i\), let \(\alpha_i\) be the hyperplane class of the
\((M|F_i,\B|F_i)\)-factor, viewed in \(\A^1(Z_F)\).

The map
\[
   \iota_F^*:\A(M,\B)\longrightarrow\A(Z_F)
\]
is the Chow-ring restriction map induced by the toric orbit closure
\(V(\sigma_F^\B)\), followed by the star-fan identification.  We use the
same notation for the induced pullback of line-bundle classes in \(K\)-theory.
For \(i=1,\ldots,\ell\), set
\[
   n_i:=\iota_F^*x_{F_i}^\B\in\A^1(Z_F),
\]
and define
\begin{equation}\label{eq:NF-def}
   N_F:=\bigoplus_{i=1}^{\ell}\OO_{Z_F}(n_i)\in\K(Z_F).
\end{equation}

\begin{lemma}[Restriction of divisor classes to the star]
\label{lem:star-pullback-divisors}
The restriction map \(\iota_F^*:\A(M,\B)\to\A(Z_F)\) satisfies
\[
   \iota_F^*x_G^\B
   =
   \begin{cases}
      x_G^L, & G\subsetneq F_i\text{ for some }i,\\[3pt]
      n_i, & G=F_i,\\[3pt]
      x_G^R, & F\subsetneq G\subsetneq E,\\[3pt]
      0, & \text{otherwise.}
   \end{cases}
\]
Moreover,
\[
   \iota_F^*\alpha=\alpha^R,
\]
and, for every \(i=1,\ldots,\ell\),
\begin{equation}\label{eq:alpha-star-relation}
   \alpha_i+n_i+
   \sum_{\substack{H\in\B\\F\subsetneq H\subsetneq E}}x_H^R
   =
   \alpha^R.
\end{equation}
\end{lemma}

\begin{proof}
The first formula is the usual restriction rule for torus-invariant divisors
to an orbit closure.  A divisor indexed by \(G\) restricts nontrivially to
\(V(\sigma_F^\B)\) exactly when the ray \(v_G\), together with
\(\sigma_F^\B\), spans a cone of the nested-set fan.  Under the closed-star
product, the surviving rays are precisely of the following types:
\[
   G\subsetneq F_i,
   \qquad
   G=F_i,
   \qquad
   F\subsetneq G\subsetneq E.
\]
They become, respectively, a divisor on the left factor, a normal line class
\(n_i\), and a divisor on the right factor.

To compute \(\iota_F^*\alpha\), choose \(e\in E\setminus F\).  The class
\(\alpha\) is represented by
\[
   \alpha=\sum_{\substack{H\in\B\\e\in H}}x_H^\B.
\]
After restriction to the star, the surviving terms are exactly those with
\(F\subsetneq H\subsetneq E\), giving \(\iota_F^*\alpha=\alpha^R\).

Now choose \(a_i\in F_i\).  The same class \(\alpha\) is represented by
\[
   \alpha=\sum_{\substack{H\in\B\\a_i\in H}}x_H^\B.
\]
After restriction to the star, the surviving terms are: the terms
\(H\subsetneq F_i\), which sum to \(\alpha_i\); the term \(H=F_i\), which
gives \(n_i\); and the terms \(F\subsetneq H\subsetneq E\), which give the
right-factor sum in \eqref{eq:alpha-star-relation}.  Comparing with
\(\iota_F^*\alpha=\alpha^R\) proves the relation.
\end{proof}

\begin{lemma}[Restriction of cutting classes to the star]
\label{lem:star-pullback-cutting}
The following identities hold in \(\A^1(Z_F)\):
\[
   \iota_F^*\theta_G^\B
   =
   \begin{cases}
      \theta_G^L, & G\subsetneq F_i\text{ for some }i,\\[3pt]
      \alpha_i, & G=F_i,\\[3pt]
      \theta_G^R, & F\subsetneq G\subsetneq E.
   \end{cases}
\]
\end{lemma}

\begin{proof}
We use the definition
\[
   \theta_G^\B
   =
   \alpha-
   \sum_{\substack{H\in\B\\G\subseteq H\subsetneq E}}x_H^\B.
\]
If \(G\subsetneq F_i\), then the surviving terms in the restriction of the
sum are
\[
   \sum_{\substack{H\in\B\\G\subseteq H\subsetneq F_i}}x_H^L,
   \qquad
   n_i,
   \qquad
   \sum_{\substack{H\in\B\\F\subsetneq H\subsetneq E}}x_H^R.
\]
Therefore
\[
\begin{aligned}
   \iota_F^*\theta_G^\B
   &=
   \alpha^R
   -
   \sum_{\substack{H\in\B\\G\subseteq H\subsetneq F_i}}x_H^L
   -n_i
   -
   \sum_{\substack{H\in\B\\F\subsetneq H\subsetneq E}}x_H^R  \\
   &=
   \alpha_i
   -
   \sum_{\substack{H\in\B\\G\subseteq H\subsetneq F_i}}x_H^L
   =
   \theta_G^L,
\end{aligned}
\]
where the second equality uses \eqref{eq:alpha-star-relation}.  If
\(G=F_i\), the same calculation gives
\[
   \iota_F^*\theta_{F_i}^\B
   =
   \alpha^R-n_i-
   \sum_{\substack{H\in\B\\F\subsetneq H\subsetneq E}}x_H^R
   =
   \alpha_i.
\]
Finally, if \(F\subsetneq G\subsetneq E\), the surviving terms are exactly the
right-factor terms \(x_H^R\) with \(G\subseteq H\subsetneq E\), and hence
\[
   \iota_F^*\theta_G^\B
   =
   \alpha^R-
   \sum_{\substack{H\in\B\\G\subseteq H\subsetneq E}}x_H^R
   =
   \theta_G^R.
\]
\end{proof}

\begin{lemma}[Right-factor telescoping]
\label{lem:right-factor-telescoping}
Let
\[
   \mathcal C_F:=\{G\in\B\mid F\subsetneq G\subsetneq E\},
   \qquad
   S_F^R:=\sum_{G\in\mathcal C_F}x_G^R.
\]
Then, in \(\K(Z_F)\),
\begin{equation}\label{eq:right-factor-telescoping}
   [\OO(\alpha^R)]
   +
   \sum_{G\in\mathcal C_F}
      \left(
         [\OO(\theta_G^R)]-[\OO(\theta_G^R+x_G^R)]
      \right)
   =
   [\OO(\alpha^R-S_F^R)].
\end{equation}
Moreover, for every \(i=1,\ldots,\ell\),
\begin{equation}\label{eq:right-factor-telescoping-alpha-i}
   [\OO(\alpha^R-S_F^R)]=[\OO(\alpha_i+n_i)].
\end{equation}
\end{lemma}

\begin{proof}
Set
\[
   m_G:=[\OO(-x_G^R)]
   \qquad (G\in\mathcal C_F).
\]
If \(G,H\in\mathcal C_F\) are incomparable, then \(G/F\) and \(H/F\) are not
\(\B/F\)-nested.  Hence the \(K\)-theoretic Stanley--Reisner relation gives
\[
   (1-m_G)(1-m_H)=0.
\]
Consequently, any product \(\prod_{G\in S}(1-m_G)\) vanishes unless
\(S\subseteq\mathcal C_F\) is a chain.

We claim that
\begin{equation}\label{eq:chain-telescope}
   1-\prod_{G\in\mathcal C_F}m_G
   =
   \sum_{G\in\mathcal C_F}
      (1-m_G)
      \prod_{\substack{H\in\mathcal C_F\\G\subsetneq H}}m_H.
\end{equation}
To prove this, write \(u_G:=1-m_G\), so \(m_G=1-u_G\).  Expanding the
left-hand side gives
\[
   \sum_{\emptyset\neq S\subseteq\mathcal C_F}
      (-1)^{|S|+1}\prod_{G\in S}u_G.
\]
All non-chain terms vanish.  On the right-hand side, the summand indexed by
\(G\) is
\[
   u_G\prod_{G\subsetneq H}(1-u_H).
\]
After discarding non-chain terms, it contributes exactly the chain terms
whose minimal element is \(G\), with coefficient \((-1)^{|S|-1}\).  Every
nonempty chain has a unique minimal element, so the two expansions agree.

Multiplying \eqref{eq:chain-telescope} by \([\OO(\alpha^R)]\), and using
\[
   \theta_G^R
   =
   \alpha^R-
   \sum_{\substack{H\in\mathcal C_F\\G\subseteq H}}x_H^R,
   \qquad
   \theta_G^R+x_G^R
   =
   \alpha^R-
   \sum_{\substack{H\in\mathcal C_F\\G\subsetneq H}}x_H^R,
\]
gives
\[
   [\OO(\alpha^R)]-[\OO(\alpha^R-S_F^R)]
   =
   \sum_{G\in\mathcal C_F}
      \left(
         [\OO(\theta_G^R+x_G^R)]-[\OO(\theta_G^R)]
      \right).
\]
Rearranging gives \eqref{eq:right-factor-telescoping}.  Finally,
\eqref{eq:alpha-star-relation} gives \(\alpha_i+n_i=\alpha^R-S_F^R\), and
therefore \eqref{eq:right-factor-telescoping-alpha-i}.
\end{proof}

\begin{definition}[Center tangent class]
\label{def:center-tangent-class}
We define the center tangent class by
\[
   \widehat T_{Z_F}:=\iota_F^*T_{M,\B}-N_F\in\K(Z_F).
\]
\end{definition}

\begin{proposition}[Star factorization of the center tangent class]
\label{prop:center-tangent-factorization}
Under the star-fan identification,
\[
   \widehat T_{Z_F}
   =
   T_{M|F,\B|F}\boxplus T_{M/F,\B/F}.
\]
Equivalently,
\begin{equation}\label{eq:star-normal}
   \iota_F^*T_{M,\B}
   =
   T_{M|F,\B|F}\boxplus T_{M/F,\B/F}+N_F.
\end{equation}
\end{proposition}

\begin{proof}
It is enough to prove \eqref{eq:star-normal}.  Set
\[
   f_i:=\rk(F_i),
   \qquad
   f:=\rk(F)=\sum_{i=1}^{\ell}f_i,
   \qquad
   r:=\rk(E).
\]
We expand \(\iota_F^*T_{M,\B}\) from \eqref{eq:intrinsic-T-def} and group
the summands according to the position of \(G\in\B\setminus\{E\}\) relative
to \(F\).

The Euler term restricts as
\begin{equation}\label{eq:euler-split-star-normal}
   \iota_F^*\bigl(r[\OO(\alpha)]-[\OO]\bigr)
   =
   r[\OO(\alpha^R)]-[\OO]
   =
   \bigl((r-f)[\OO(\alpha^R)]-[\OO]\bigr)
   +
   f[\OO(\alpha^R)].
\end{equation}
The term in parentheses is the Euler term of \(T_{M/F,\B/F}\).

If \(G\subsetneq F_i\), Lemmas~\ref{lem:star-pullback-divisors} and
\ref{lem:star-pullback-cutting} give
\[
   \iota_F^*x_G^\B=x_G^L,
   \qquad
   \iota_F^*\theta_G^\B=\theta_G^L.
\]
All such summands give the non-Euler part of \(T_{M|F,\B|F}\).

If \(G=F_i\), then
\[
   \iota_F^*x_{F_i}^\B=n_i,
   \qquad
   \iota_F^*\theta_{F_i}^\B=\alpha_i.
\]
The \(F_i\)-summand contributes
\[
   [\OO(n_i)]-[\OO]
   +
   f_i\left([\OO(\alpha_i)]-[\OO(\alpha_i+n_i)]\right).
\]
Equivalently, this is
\[
   \bigl(f_i[\OO(\alpha_i)]-[\OO]\bigr)
   +
   [\OO(n_i)]
   -
   f_i[\OO(\alpha_i+n_i)].
\]
After summing over \(i\), the first terms give the Euler part of
\(T_{M|F,\B|F}\), the middle terms give \(N_F\), and the leftover is
\begin{equation}\label{eq:leftover-factor-terms}
   -\sum_{i=1}^{\ell}f_i[\OO(\alpha_i+n_i)].
\end{equation}

If \(F\subsetneq G\subsetneq E\), then
\[
   \iota_F^*x_G^\B=x_G^R,
   \qquad
   \iota_F^*\theta_G^\B=\theta_G^R.
\]
Moreover,
\[
   \rk_M(G)=f+\rk_{M/F}(G/F).
\]
Thus the \(\rk_{M/F}(G/F)\)-part, together with the terms
\([\OO(x_G^R)]-[\OO]\), is the non-Euler part of \(T_{M/F,\B/F}\).  The
remaining rank-shift contribution is
\begin{equation}\label{eq:rank-shift-star-normal}
   f
   \sum_{G\in\mathcal C_F}
      \left(
         [\OO(\theta_G^R)]-[\OO(\theta_G^R+x_G^R)]
      \right).
\end{equation}

Finally, if \(G\) is neither contained in \(F\) nor contains \(F\), then
\(\iota_F^*x_G^\B=0\), and the corresponding summand restricts to zero.

It remains to cancel the extra terms.  Combining the second term in
\eqref{eq:euler-split-star-normal} with \eqref{eq:rank-shift-star-normal},
Lemma~\ref{lem:right-factor-telescoping} gives
\[
\begin{aligned}
   &f[\OO(\alpha^R)]
   +
   f\sum_{G\in\mathcal C_F}
      \left(
         [\OO(\theta_G^R)]-[\OO(\theta_G^R+x_G^R)]
      \right) \\
   &\hspace{3cm}=
   f[\OO(\alpha^R-S_F^R)].
\end{aligned}
\]
By \eqref{eq:right-factor-telescoping-alpha-i}, this is equal to
\(f[\OO(\alpha_i+n_i)]\) for every \(i\).  Since \(f=\sum_i f_i\), it is
also
\[
   \sum_{i=1}^{\ell}f_i[\OO(\alpha_i+n_i)],
\]
which cancels exactly with \eqref{eq:leftover-factor-terms}.  The remaining
terms are precisely
\[
   T_{M|F,\B|F}\boxplus T_{M/F,\B/F}+N_F.
\]
This proves \eqref{eq:star-normal}, and the stated factorization follows
from the definition of \(\widehat T_{Z_F}\).
\end{proof}

\subsection{Proof of tangent realization}

\begin{proof}[Proof of Theorem~\ref{thm:tangent-realization}]
We prove the theorem by descending induction along a one-step chain
\[
   \G=\G_0\subset\G_1\subset\cdots\subset\G_N=\Gmax.
\]
For \(\G_N=\Gmax\), the equality
\[
   \Hy(M,\Gmax)=\Phi_y(T_{M,\Gmax})
\]
is the definition \eqref{eq:Hy-max-def}.

Assume the result is known for \(\B^+=\B\cup\{F\}\) and for the induced
building sets on \(M|F\) and \(M/F\).  Proposition~\ref{prop:relative-Aluffi}
gives the Aluffi \(K\)-recursion for
\(\rho_F:X_{M,\B^+}\to X_{M,\B}\), and
Proposition~\ref{prop:center-tangent-factorization} gives
\[
   T_{M|F,\B|F}\boxplus T_{M/F,\B/F}
   =
   \iota_F^*T_{M,\B}-N_F.
\]
Thus the hypotheses of Lemma~\ref{lem:universal-one-step-HRR} are satisfied
with
\[
   \T_X=T_{M,\B},
   \qquad
   \T_{\widetilde X}=T_{M,\B^+},
   \qquad
   \T_Z=T_{M|F,\B|F}\boxplus T_{M/F,\B/F}.
\]
The lemma gives
\[
\begin{aligned}
   \Phi_y(T_{M,\B})
   =&\;(\rho_F)_*\Phi_y(T_{M,\B^+})        \\
   &-
   q_\ell(y)(\iota_F)_*
   \Phi_y(T_{M|F,\B|F}\boxplus T_{M/F,\B/F}).
\end{aligned}
\]
By multiplicativity of \(\Phi_y\),
\[
   \Phi_y(T_{M|F,\B|F}\boxplus T_{M/F,\B/F})
   =
   \Phi_y(T_{M|F,\B|F})\boxtimes\Phi_y(T_{M/F,\B/F}).
\]
Using the induction hypothesis on \((M,\B^+)\), \((M|F,\B|F)\), and
\((M/F,\B/F)\), the preceding identity is exactly the recursive definition
\eqref{eq:Hy-recursion} of \(\Hy(M,\B)\).  Hence
\[
   \Phi_y(T_{M,\B})=\Hy(M,\B),
\]
and the induction is complete.

Setting \(y=0\) gives \(\Phi_0(T_{M,\G})=\td(T_{M,\G})\).  On the other hand,
\(q_\ell(0)=0\), so the recursive definition of \(\Hy\) at \(y=0\) is the
pushforward definition \eqref{eq:formal-Todd-def} of \(\Td_{M,\G}\).  Thus
\(\Td_{M,\G}=\td(T_{M,\G})\).  The Chow-polynomial formula follows from
Proposition~\ref{prop:formal-Chow-polynomial}.
\end{proof}

Theorem \ref{thm:tangent-realization} gives the promised Hirzebruch-genus interpretation of the Chow polynomial.
\begin{corollary}[Chow polynomial as a Hirzebruch genus]
\label{cor:tangent-Chow-polynomial}
For every building set \(\G\),
\[
   \deg_{M,\G}
   \left(
      \ch(\lambda_yT_{M,\G}^\vee)\td(T_{M,\G})
   \right)\bigg|_{y=-t}
   =
   H_M^\G(t).
\]
Equivalently,
\[
   \dim_\Q\A^p(M,\G)
   =
   (-1)^p
   \deg_{M,\G}
   \left(
      [y^p]\,\ch(\lambda_yT_{M,\G}^\vee)\td(T_{M,\G})
   \right).
\]
\end{corollary}

\begin{proof}
This is Theorem~\ref{thm:tangent-realization} together with
Proposition~\ref{prop:formal-Chow-polynomial}.
\end{proof}
\begin{remark}
\label{rem:gamma-coefficients}
The standard expansion of the Hirzebruch \(\chi_y\)-genus around \(y = -1\) (see \cite{NR_Prym_75} or \cite[Lemma~4.1]{ChengLinInequalities}) relates the Chow polynomial and the Chern classes of \(T_{M, \G}\). In particular, the \(\gamma\)-coefficient can be expressed by the Chern numbers (see \cite[Remark 3.35]{ChengLinInequalities}).
\end{remark}

\subsection{Numerical consequence from the geometric picture}
\label{subsec:consequences}

Let \(d:=r-1\). 
\begin{corollary}[Formal canonical class]
\label{cor:canonical-class}
The formal canonical class associated to the tangent class is
\begin{equation}\label{eq:formal-canonical-class}
   K_{M,\G}:=-c_1(T_{M,\G})
   =
   -r\alpha+
   \sum_{F\in\G\setminus\{E\}}(\rk(F)-1)x_F.
\end{equation}
Equivalently,
\[
   \omega_{M,\G}:=\OO(K_{M,\G})=\det(T_{M,\G}^\vee).
\]
Moreover, if \(\B^+=\B\cup\{F\}\) is a one-step enlargement and
\(\ell=|\Fac_\B(F)|\), then
\[
   K_{M,\B^+}=\rho_F^*K_{M,\B}+(\ell-1)x_F.
\]
\end{corollary}

\begin{proof}
The formula for \(K_{M,\G}\) follows by taking first Chern classes in
Definition~\ref{def:tangent}.  For the one-step formula, take first Chern
classes in Proposition~\ref{prop:relative-Aluffi}.  The relative correction
has first Chern class
\[
   x_F+
   \sum_{i=1}^{\ell}
      \bigl(x_{F_i}^{\B^+}-(x_{F_i}^{\B^+}+x_F)\bigr)
   =
   -(\ell-1)x_F.
\]
Thus
\[
   c_1(T_{M,\B^+})=
   \rho_F^*c_1(T_{M,\B})-(\ell-1)x_F,
\]
and the canonical formula follows by multiplying by \(-1\).
\end{proof}

\begin{corollary}[Formal Serre duality]
\label{cor:formal-serre-duality}
For every \(\xi\in\K(M,\G)\),
\[
   \chi_{M,\G}(\xi)
   =
   (-1)^d
   \chi_{M,\G}
   \left(
      \xi^\vee\otimes\omega_{M,\G}
   \right).
\]
\end{corollary}

\begin{proof}
The proof is almost the same as the proof of \cite[Corollary 3.22]{ChengTangent}. By Theorem~\ref{thm:tangent-realization} and Proposition~\ref{prop:formal-HRR},
\[
   \chi_{M,\G}(\xi)
   =
   \int_{M,\G}\ch(\xi)\td(T_{M,\G}).
\]
Write \(T=T_{M,\G}\), and let \(u_1,\ldots,u_d\) be its Chern roots.  Then
\[
   \ch(\omega_{M,\G})=e^{-c_1(T)},
   \qquad
   e^{-c_1(T)}\td(T)=\prod_{a=1}^d\frac{u_a}{e^{u_a}-1}.
\]
The degree-\(k\) part of \(\ch(\xi^\vee)\) is \((-1)^k\) times the degree-
\(k\) part of \(\ch(\xi)\), and the last displayed product is obtained from
\(\td(T)\) by changing \(u_a\) to \(-u_a\).  Hence the integrand for
\(\xi^\vee\otimes\omega_{M,\G}\) is the degree-dual of the integrand for
\(\xi\).  Taking the degree-\(d\) part introduces the factor \((-1)^d\),
which proves the claim.
\end{proof}

We next record several numerical consequences.

\begin{corollary}[Hyperplane degrees of the Todd class]
\label{cor:Todd-alpha-degrees}
Write
\[
   \td(T_{M,\G})=\sum_{k=0}^d\td_k(T_{M,\G}).
\]
Then
\[
   \sum_{k=0}^d
   \frac{m^{d-k}}{(d-k)!}
   \int_{M,\G}\alpha^{d-k}\td_k(T_{M,\G})
   =
   \binom{m+d}{d}.
\]
Equivalently,
\[
   \int_{M,\G}\alpha^{d-k}\td_k(T_{M,\G})
   =
   (d-k)!\,[m^{d-k}]\binom{m+d}{d}.
\]
\end{corollary}

\begin{proof}
By definition of \(\chi_{M,\G}\), and since
\[
   \rho_{\Gmax,\G}^*\alpha=\alpha,
\]
we have
\[
   \chi_{M,\G}(\OO(m\alpha))
   =
   \chi_{M,\Gmax}(\OO(m\alpha)).
\]
For the maximal building set, the tangent-bundle calculation gives
\[
   \chi_{M,\Gmax}(\OO(m\alpha))=\binom{m+d}{d};
\]
see \cite{ChengTangent}.  On the other hand, formal HRR and
Theorem~\ref{thm:tangent-realization} give
\[
   \chi_{M,\G}(\OO(m\alpha))
   =
   \int_{M,\G}e^{m\alpha}\td(T_{M,\G}).
\]
Expanding \(e^{m\alpha}\) and comparing coefficients of \(m\) proves the
claim.
\end{proof}

\begin{lemma}
\label{lem:xF-power-alpha}
Let \(S=\{F_1,\ldots,F_\ell\}\subseteq\G\setminus\{E\}\), and set
\[
   F:=F_1\vee\cdots\vee F_\ell.
\]
If \(m<\rk(F)\), then
\[
   x_{F_1}\cdots x_{F_\ell}\alpha^{d-m}=0.
\]
Moreover, if \(F\in\G\setminus\{E\}\) has rank \(s:=\rk(F)\), then
\[
   \int_{M,\G}x_F^s\alpha^{d-s}
   =
   (-1)^{s-1}.
\]
\end{lemma}

\begin{proof}
We may assume that \(S=\{F_1,\ldots,F_\ell\}\subseteq\G\setminus\{E\}\) be a
\(\G\)-nested set. Let
\[
   \sigma_S:=\cone(v_{F_1},\ldots,v_{F_\ell}),
   \qquad
   \iota_S:V(\sigma_S)\hookrightarrow X_{M,\G}
\]
be the orbit closure corresponding to the nested set \(S\).  The standard
toric orbit-closure formula gives
\[
   (\iota_S)_*[1]=x_{F_1}\cdots x_{F_\ell}.
\]
Hence, by the projection formula,
\[
   x_{F_1}\cdots x_{F_\ell}\alpha^{d-m}
   =
   (\iota_S)_*\iota_S^*(\alpha^{d-m}).
\]
The star of \(\sigma_S\) is the product of the nested-set fans associated to
the intervals cut out by \(S\).  In particular, the factor above
\(F=F_1\vee\cdots\vee F_\ell\) is the contraction factor \(M/F\), and
\(\iota_S^*\alpha\) is pulled back from the hyperplane class on this factor.
Since \(M/F\) has rank \(r-\rk(F)\), this factor has top Chow degree
\[
   r-\rk(F)-1=d-\rk(F).
\]
If \(m<\rk(F)\), then \(d-m>d-\rk(F)\), so
\[
   \iota_S^*(\alpha^{d-m})=0.
\]
This proves the vanishing.

It remains to prove the self-intersection formula.  Let
\[
   \iota_F:D_F\hookrightarrow X_{M,\G}
\]
be the boundary divisor corresponding to the ray \(v_F\).  The star of
\(v_F\) is the restriction--contraction product, so
\[
   \A(D_F)
   \simeq
   \A(M|F,\G|F)\otimes \A(M/F,\G/F).
\]
Let \(\alpha^L\) and \(\alpha^R\) denote the hyperplane classes on the
restriction and contraction factors.  The divisor-restriction rule gives
\[
   \iota_F^*\alpha=\alpha^R.
\]
Moreover,
\[
   \iota_F^*x_F+\alpha^L+S_F^R=\alpha^R,
\]
where \(S_F^R\) belongs to the contraction factor.  Thus
\[
   \iota_F^*x_F=\alpha^R-\alpha^L-S_F^R.
\]
By the projection formula,
\[
   \int_{M,\G}x_F^s\alpha^{d-s}
   =
   \int_{D_F}(\iota_F^*x_F)^{s-1}(\alpha^R)^{d-s}.
\]
After multiplying by \((\alpha^R)^{d-s}\), every term involving either
\(\alpha^R\) or a class from \(S_F^R\) has degree greater than \(d-s\) on
the contraction factor, and therefore vanishes.  Hence
\[
   (\iota_F^*x_F)^{s-1}(\alpha^R)^{d-s}
   =
   (-\alpha^L)^{s-1}(\alpha^R)^{d-s} = (-1)^{s-1}.
\]
\end{proof}

\begin{lemma}[Truncated building set]
\label{lem:truncated-building-set}
Let \(0\le k\le d\), and let
\[
   M^{(k)}:=\operatorname{Tr}^{d-k}M
\]
be the \((d-k)\)-fold truncation of \(M\), so \(M^{(k)}\) has rank \(k+1\).
Set
\[
   \G^{(k)}
   :=
   \{F\in\G\mid \rk_M(F)\le k\}\cup\{E\}.
\]
Then \(\G^{(k)}\) is a building set in the lattice of flats of \(M^{(k)}\).
\end{lemma}

\begin{proof}
The non-maximal flats of \(M^{(k)}\) are exactly the flats of \(M\) of rank at most \(k\) (when \(k \geq 1\)), and the interval below such a flat is unchanged by truncation.
Thus, if \(X\neq E\) is a flat of \(M^{(k)}\), the
\(\G^{(k)}\)-factors of \(X\) are the same as the \(\G\)-factors of \(X\),
and the required product decomposition follows from the fact that \(\G\) is
a building set. For the top flat \(E\) of \(M^{(k)}\), the unique maximal element of \(\G^{(k)}\) contained in \(E\) is \(E\) itself. Therefore, \(\G^{(k)}\) is a building set of \(M^{(k)}\).
\end{proof}

\begin{lemma}[\(\alpha\)-truncation in degree]
\label{lem:alpha-truncation}
Let \(0\le k\le d\), and use the notation of
Lemma~\ref{lem:truncated-building-set}.  If \(P\) is a homogeneous
degree-\(k\) polynomial in \(\alpha\) and in the variables \(x_F\) with
\(\rk_M(F)\le k\), then
\[
   \int_{M,\G}P\,\alpha^{d-k}
   =
   \int_{M^{(k)},\G^{(k)}}P.
\]
\end{lemma}

\begin{proof}
Consider the polynomial ring on the variables \(x_F\) with
\(F\in\G\setminus\{E\}\).  Sending
\[
   x_F\longmapsto
   \begin{cases}
      x_F, & \rk_M(F)\le k,\\
      0, & \rk_M(F)>k
   \end{cases}
\]
and sending \(\alpha\) to \(\alpha\) is compatible with the linear relations
in the Feichtner--Yuzvinsky presentations.  Therefore it gives a natural
surjective homomorphism
\[
   \pi_k:\A(M,\G)\longrightarrow \A(M^{(k)},\G^{(k)}).
\]

We claim that the degree functional
\[
   L_k(P):=\int_{M,\G}P\,\alpha^{d-k}
\]
on homogeneous degree-\(k\) polynomials in the variables of rank at most
\(k\) factors through \(\pi_k\).  In other words, if such a polynomial
represents zero in \(\A(M^{(k)},\G^{(k)})\), then its product with
\(\alpha^{d-k}\) has degree zero in \(\A(M,\G)\).

It is enough to check the additional Stanley--Reisner relations that appear
after truncation.  Let
\[
   S_1,\ldots,S_m\in \G^{(k)}\setminus\{E\}
\]
be pairwise incomparable, and suppose that
\[
   S_1\vee_{M^{(k)}}\cdots\vee_{M^{(k)}}S_m
   \in \G^{(k)}.
\]
If the join
\[
   J:=S_1\vee_M\cdots\vee_M S_m
\]
has \(\rk_M(J) > k\), Lemma~\ref{lem:xF-power-alpha} implies \(x_{S_1}\cdots x_{S_m}\,\alpha^{d-k} = 0\). When \(\rk_M(J) \leq k\), the join is the same in \(M\) and in \(M^{(k)}\).  In this case \(J\in\G\), so
\[
   x_{S_1}\cdots x_{S_m}=0
   \qquad\text{already in }\A(M,\G).
\]
Thus the relation creates nothing new, and \(L_k\) factors through \(\pi_k\).

Therefore, \(L_k\) descends to a linear functional on \(\A^k(M^{(k)},\G^{(k)})\). The top degree of this Chow ring is one-dimensional, so the descended
functional is a scalar multiple of the degree map. The scalar is determined by
\[
   L_k(\alpha^k)
   =
   \int_{M,\G}\alpha^d
   =
   1.
\] Consequently, the descended functional is the usual degree map.
\end{proof}

\begin{proposition}[Tangent class and truncation]
\label{prop:tangent-truncation}
With the notation above,
\[
   \int_{M,\G}c_k(T_{M,\G})\alpha^{d-k}
   =
   \int_{M^{(k)},\G^{(k)}}
      \left[
         (1+\alpha)^{d-k}
         c(T_{M^{(k)},\G^{(k)}})
      \right]_k .
\]
\end{proposition}

\begin{proof}
Use the product formula
\[
   c(T_{M,\G})
   =
   (1+\alpha)^r
   \prod_{F\in\G\setminus\{E\}}
   (1+x_F)
   \left(
      \frac{1+\theta_F^\G}{1+\theta_F^\G+x_F}
   \right)^{\rk(F)}.
\]
Any positive-degree term coming from a flat \(F\) with \(\rk(F)>k\) contains
a factor \(x_F\).  Such a term vanishes after multiplication by
\(\alpha^{d-k}\), by Lemma~\ref{lem:xF-power-alpha}.

Thus only flats of rank at most \(k\) contribute.  For such flats, the
projection \(\pi_k\) sends
\[
   \theta_F^\G
   =
   \alpha-
   \sum_{\substack{H\in\G\\F\subseteq H\subsetneq E}}x_H \qquad \text{to} \qquad
   \alpha-
   \sum_{\substack{H\in\G\\F\subseteq H\subsetneq E\\ \rk(H)\le k}}x_H
   =
   \theta_F^{\G^{(k)}}.
\]
Moreover,
\[
   (1+\alpha)^r=(1+\alpha)^{k+1}(1+\alpha)^{d-k},
\]
and \((1+\alpha)^{k+1}\) is the Euler factor for
\(T_{M^{(k)},\G^{(k)}}\).  Applying Lemma~\ref{lem:alpha-truncation} gives
the stated identity.
\end{proof}

\begin{corollary}[Chern numbers against \(\alpha\)]
\label{cor:Chern-alpha-positive}
For every building set \(\G\) and every \(0\le k\le d\),
\[
   \int_{M,\G}c_k(T_{M,\G})\alpha^{d-k}
   \ge
   \binom{d+1}{k}.
\]
Equivalently, the \(k\)-th Chern number against \(\alpha^{d-k}\) is at least
the corresponding value for \(\PP^d\).
\end{corollary}

\begin{proof}
Set
\[
   P_k(M,\G)
   :=
   \int_{M,\G}c_k(T_{M,\G})\alpha^{d-k}.
\]
By Proposition~\ref{prop:tangent-truncation},
\[
   P_k(M,\G)
   =
   \sum_{j=0}^k
      \binom{d-k}{k-j}
      P_j(M^{(k)},\G^{(k)}).
\]
We prove the claim by induction on \(k\).  The case \(k=0\) is immediate.

Assume the result is known for all smaller indices.  Since \(M^{(k)}\) has
rank \(k+1\), the induction hypothesis gives
\[
   P_j(M^{(k)},\G^{(k)})
   \ge
   \binom{k+1}{j}
   \qquad
   (j<k).
\]
For \(j=k\), Theorem~\ref{thm:tangent-realization} at \(y=-1\) gives
\[
   P_k(M^{(k)},\G^{(k)})
   =
   \int c_k(T_{M^{(k)},\G^{(k)}})
   =
   H_{M^{(k)}}^{\G^{(k)}}(1).
\]
Since \(\int_{M^{(k)},\G^{(k)}}\alpha^k=1\), the Chow ring has nonzero
classes in degrees \(0,\ldots,k\).  Therefore
\[
   H_{M^{(k)}}^{\G^{(k)}}(1)
   =
   \sum_{p=0}^k\dim_\Q\A^p(M^{(k)},\G^{(k)})
   \ge
   k+1
   =
   \binom{k+1}{k}.
\]
Thus
\[
\begin{aligned}
   P_k(M,\G)
   \ge
   \sum_{j=0}^k
      \binom{d-k}{k-j}
      \binom{k+1}{j} = \binom{d+1}{k},
\end{aligned}
\]
by Vandermonde's identity.
\end{proof}
\begin{corollary}[Monotonicity in the building set]
\label{cor:Chern-alpha-monotone}
Let \(\G\subseteq \Hh\) be building sets containing \(E\).  Then, for every
\(0\le k\le d\),
\[
   \int_{M,\Hh}c_k(T_{M,\Hh})\alpha^{d-k}
   \ge
   \int_{M,\G}c_k(T_{M,\G})\alpha^{d-k}.
\]
\end{corollary}

\begin{proof}
This follows from the proof of Corollary~\ref{cor:Chern-alpha-positive}.
Indeed, the same induction applies after observing that
\[
   \G^{(k)}\subseteq \Hh^{(k)}
\]
for every \(k\), and that the top-degree term is monotone by the one-step
blow-up formula for the Chow polynomial.
\end{proof}
\begin{remark}
Assume that \(M\) is realizable.  The preceding arguments have simple
geometric interpretations since \(\alpha\) is the pullback of the hyperplane class from the original projective space before the wonderful blow-ups.

If \(F\) has rank \(s\), then the center indexed by \(F\) has image of
dimension \(d-s\) in the original projective space.  Hence \(q>d-s\) general
hyperplanes may be chosen to avoid this image. The formula
\[
   \int x_F^s\alpha^{d-s}=(-1)^{s-1}
\]
is the corresponding exceptional-divisor self-intersection.  

The truncation statement has the same geometric meaning: intersecting with
\(d-k\) general hyperplanes cuts the original arrangement to a rank
\(k+1\) truncation (see \cite{BES24}).  The factor \((1+\alpha)^{d-k}\) in
Proposition~\ref{prop:tangent-truncation} records the normal bundle of this
linear section. The final inequality now follows from tracking topological Euler characteristic.
\end{remark}

\begin{corollary}[Miyaoka--Yau inequality]
\label{cor:miyaoka-yau}
Assume \(d\ge2\), then
\begin{equation}\label{eq:BMY-alpha-inequality}
   d
   \int_{M,\G}c_1(T_{M,\G})^2\alpha^{d-2}
   \le
   2(d+1)
   \int_{M,\G}c_2(T_{M,\G})\alpha^{d-2}.
\end{equation}
This has the same form as the higher-dimensional
Miyaoka--Yau inequality with respect to the hyperplane class \(\alpha\); compare \cite{GrebKebekusTaji}.
\end{corollary}

\begin{proof}
The proof follows from combining Corollary~\ref{cor:Chern-alpha-positive} and Corollary~\ref{cor:Todd-alpha-degrees} when \(k = 2\).
\end{proof}
In fact, a direct computation will show the difference \[   
    d\int_{M,\G}c_1(T_{M,\G})^2\alpha^{d-2} -
    2(d+1)\int_{M,\G}c_2(T_{M,\G})\alpha^{d-2} = -(3d+2)n_2,
\] where \(n_2\) is the number of rank 2 flats in \(\G\).

\printbibliography

\end{document}